\documentclass[reqno, oneside]{amsart}
\usepackage{etex}
\usepackage{hyperref}

\usepackage{chemarr}
\usepackage{amssymb}
\usepackage{comment}

\usepackage[a4paper]{geometry}

\theoremstyle{definition}
\newtheorem{nul}{}[section]
\newtheorem{dfn}[nul]{Definition}

\newtheorem{rmk}[nul]{Remark}

\newtheorem{cnstr}[nul]{Construction}

\newtheorem{exm}[nul]{Example}
\newtheorem{obs}[nul]{Observation}

\newtheorem{qst}{Question}
\newtheorem*{dfn*}{Definition}
\newtheorem*{axm*}{Axiom}
\newtheorem*{ntn*}{Notation}
\newtheorem*{exm*}{Example}
\newtheorem*{exr*}{Exercise}
\newtheorem*{int*}{Intuition}
\newtheorem*{qst*}{Question}
\newtheorem*{rmk*}{Remark}

\theoremstyle{plain}

\newtheorem{thm}[nul]{Theorem}
\newtheorem{prop}[nul]{Proposition}
\newtheorem{lem}[nul]{Lemma}
\newtheorem{var}[nul]{Variant}

\newtheorem{cor}{Corollary}[nul]
\newtheorem*{thm*}{Theorem}
\newtheorem*{prop*}{Proposition}
\newtheorem*{cor*}{Corollary}
\newtheorem*{lem*}{Lemma}
\newtheorem*{cnj*}{Conjecture}
\newtheorem{innercustomthm}{Theorem}
\newenvironment{customthm}[1]
  {\renewcommand\theinnercustomthm{#1}\innercustomthm}
  {\endinnercustomthm}

\DeclareMathOperator*{\colim}{\mathrm{colim}}

\DeclareMathOperator{\fib}{\mathrm{fib}}
\DeclareMathOperator{\Hom}{\text{Hom}}

\DeclareMathOperator{\smsh}{\wedge}

\DeclareMathOperator{\C}{\mathcal{C}}
\DeclareMathOperator{\D}{\mathcal{D}}
\DeclareMathOperator{\CP}{\mathbb{CP}}
\DeclareMathOperator{\Z}{\mathbb{Z}}
\DeclareMathOperator{\E}{\mathbb{E}}

\DeclareMathOperator{\N}{\mathrm{N}}

\DeclareMathOperator{\cS}{\mathcal{S}}

\DeclareMathOperator{\Ran}{\mathrm{Ran}}

\DeclareMathOperator{\Poly}{\text{Poly}}
\DeclareMathOperator{\Fin}{\mathrm{Fin}}

\DeclareMathOperator{\Gra}{\mathrm{Gr}}

\DeclareMathOperator{\Gr}{\mathbf{Gr}}
\DeclareMathOperator{\Fil}{\mathbf{Fil}}
\DeclareMathOperator{\Fun}{\text{Fun}}
\DeclareMathOperator{\Alg}{\mathrm{Alg}}
\DeclareMathOperator{\Sp}{\mathrm{Sp}}

\DeclareMathOperator{\J}{\mathcal{J}}
\DeclareMathOperator{\Cofil}{\mathbf{Cofil}}

\usepackage{tikz}
\usetikzlibrary{matrix,arrows,decorations}
\usepackage{tikz-cd}

\usepackage{adjustbox}

\begin{document}

\title{Multiplicative Structure in the Stable Splitting of $\Omega SL_n(\mathbb{C})$}
\author{Jeremy Hahn}
\address{Department of Mathematics, Massachusetts Institute of Technology, Cambridge, MA 02139}
\email{jhahn01@mit.edu}

\author{Allen Yuan}
\address{Department of Mathematics, Massachusetts Institute of Technology, Cambridge, MA 02139}
\email{alleny@mit.edu}

\begin{abstract}
The space of based loops in $SL_n(\mathbb{C})$, also known as the affine Grassmannian of $SL_n(\mathbb{C})$, admits an $\mathbb{E}_2$ or fusion product.  Work of Mitchell and Richter proves that this based loop space stably splits as an infinite wedge sum.  We prove that the Mitchell--Richter splitting is coherently multiplicative, but not $\mathbb{E}_2$.  Nonetheless, we show that the splitting becomes $\mathbb{E}_2$ after base-change to complex cobordism.  Our proof of the $\mathbb{A}_\infty$ splitting involves on the one hand an analysis of the multiplicative properties of Weiss calculus, and on the other a use of Beilinson--Drinfeld Grassmannians to verify a conjecture of Mahowald and Richter.  Other results are obtained by explicit, obstruction-theoretic computations.
\end{abstract}


\setcounter{tocdepth}{1}
\maketitle

\tableofcontents

\vbadness 5000


\section{Introduction}

We study the homotopy type of the affine Grassmannian of $SL_n(\mathbb{C})$, which is equivalent to the space $\Omega SU(n)$ of based loops in $SU(n)$.  There are essentially two multiplications on this homotopy type, one arising from the composition of loops and the other from the group multiplication on $SL_n(\mathbb{C})$.  Together, these two multiplications interact to give $\Omega SU(n)$ the structure of an $\mathbb{E}_2$ or chiral algebra.  In geometric representation theory, this structure is witnessed by the existence of the Beilinson--Drinfeld Grassmannian.

Either of the above (homotopy equivalent) products make $H_*(\Omega SU(n);\mathbb{Z})$ into a graded ring.  To describe this ring, let us first name some of its elements.  For each one-dimensional subspace $V \subset \mathbb{C}^n$, there is a loop $\lambda_V:S^1 \rightarrow U(n)$ given by the formula
$$\lambda_V(z)=\left( \begin{array}{cc} z & 0 \\ 0 & I \end{array} \right),$$
with the matrix presented in terms of the decomposition $\mathbb{C}^n \cong V \oplus V^{\perp}$.  Fixing a particular line $W \subset \mathbb{C}^n$, the construction $V \mapsto \lambda_W^{-1} \cdot \lambda_V$ defines a well-known map
$$\mathbb{CP}^{n-1} \rightarrow \Omega SU(n).$$
For $1 \le i \le n-1$, let $b_i \in H_{2i}(\Omega SU(n);\mathbb{Z})$ denote the image of the generator of $H_{2i}(\mathbb{CP}^{n-1};\mathbb{Z})$.  It is a result of Bott \cite{Bott} that
$$H_*(\Omega SU(n);\mathbb{Z}) \cong \mathbb{Z}[b_1,b_2,\cdots,b_{n-1}],$$
with the latter denoting the polynomial algebra on the classes $b_i$.

Notice that $H_*(\Omega SU(n);\mathbb{Z})$ is a \textit{bigraded} ring: there is, in addition to the usual homological grading $*$, a word length grading that assigns each $b_i$ degree $1$.  Mahowald observed that the action of the Steenrod algebra on $H_*(\Omega SU(n);\mathbb{F}_2)$ preserves word length, and he conjectured a geometric splitting to be responsible.

Motivated by Mahowald's conjecture, Mitchell \cite{MitchellSU(n)} (and, independently, Segal \cite{Segal}) constructed a filtration $$* =F_{n,0} \longrightarrow F_{n,1} \longrightarrow F_{n,2}\longrightarrow \cdots \longrightarrow \Omega SU(n).$$
Following Mitchell, we name this the \textit{Bott filtration} of $\Omega SU(n)$.  The homology of $F_{n,k}$ consists of words of length at most $k$, and the inclusion $F_{n,1} \rightarrow \Omega SU(n)$ is given by the above map $\mathbb{CP}^{n-1} \rightarrow \Omega SU(n)$.  For $k>1$, the homotopy type $F_{n,k}$ may be modeled as a singular algebraic variety, and there is a surjective resolution of singularities $(\mathbb{CP}^{n-1})^{\times k} \longrightarrow F_{n,k}$.  The exact construction of the Bott filtration is somewhat involved, and we review it in Section \ref{sec:MRFil}--it is a subfiltration of the Bruhat ordering on (closures of) Iwahori orbits.

Confirming Mahowald's intuition, Mitchell and Richter \cite[Theorem 2.1]{CrabbMitchell} proved that the Bott filtration splits after taking suspension spectra.  In short, there is a wedge sum decomposition
$$\Sigma^{\infty}_+ \Omega SU(n) \simeq \bigvee_k \Sigma^{\infty} F_{n,k}/F_{n,k-1} \simeq \mathbb{S} \vee \Sigma^{\infty} \mathbb{CP}^{n-1} \vee \cdots.$$

\begin{rmk}
Readers unfamiliar with stable homotopy theory may prefer to think of the suspension spectrum functor $\Sigma^\infty_+$ as analogous to taking the underlying motive of a variety.  The statement is then that this motive splits; in particular, its homology will split in any homology theory.  
\end{rmk}

\begin{exm}
In the case $n=2$, the Bott filtration of $\Omega SU(2) \simeq \Omega S^3$ is the classical James filtration of $\Omega \Sigma S^2$.  The Mitchell--Richter splitting recovers the stable James splitting.
\end{exm}

In this paper, we will be interested in \emph{multiplicative} aspects of the Bott filtration and its splitting.  In Section \ref{sec:FilGra}, we review the ($\infty$)-categories of filtered and graded spectra.  A filtered spectrum is a sequence of spectra $(X_0 \to X_1 \to X_2 \to \cdots)$ connected by maps; a graded spectrum is simply a sequence $(X_0,X_1,X_2,\cdots)$ of spectra.  These categories acquire symmetric monoidal structures by Day convolution coming from the addition of nonnegative integers.  This allows us to talk about $\E_n$-algebras in filtered and graded spectra, providing the language necessary to state our first main theorem (proven in Section \ref{sec:MRFil}):

\begin{thm} \label{thm:BottIsAoo}
The suspension of the Bott filtration 
$$\mathbb{S} \longrightarrow \Sigma_+^{\infty} \mathbb{CP}^{n-1} \simeq \Sigma_+^{\infty} F_{n,1} \longrightarrow \Sigma_+^{\infty} F_{n,2} \longrightarrow \cdots \longrightarrow \Sigma^{\infty}_+ \Omega SU(n).$$
is an $\mathbb{A}_\infty$-algebra object in filtered spectra.
\end{thm}

\begin{rmk}
The Bott filtration is multiplicative before suspension, but for technical reasons we prefer to phrase our results in terms of filtered spectra instead of filtered spaces.
\end{rmk}

\begin{qst} \label{qst:BottE2}
Is the Bott filtration an $\mathbb{E}_2$ filtration?  We do not know the answer--for some thoughts about the problem, see Remark \ref{rmk:E2fil}.
\end{qst}

The proof of Theorem \ref{thm:BottIsAoo} is fairly straightforward, once given access to the sophisticated machinery behind the Beilinson--Drinfeld Grassmannian.  For example, we will explain in Section \ref{sec:MRFil} that this machinery immediately dispenses with a conjecture of Mahowald and Richter \cite{MahowaldRichter}.  Nonetheless, there are some subtleties involved, and it is these subtleties that prevent us from determining if the Bott filtration is $\mathbb{E}_2$.  The problem is readily visible in the case $n=\infty$:

\begin{exm}\label{exm:BottFil}
The limiting case of the Bott filtration of $\Omega SU(n)$ as $n$ tends to $\infty$ is the filtration
$$* \longrightarrow BU(1) \longrightarrow BU(2) \longrightarrow BU(3) \longrightarrow \cdots \longrightarrow BU \simeq \Omega SU.$$
It is easy to see that $\coprod BU(n)$ is a graded $\mathbb{E}_2$-algebra in spaces (in fact, it is a graded $\mathbb{E}_\infty$-algebra, being the nerve of the category of vector spaces).  However, the filtered object is much more subtle.  For example, the squares
$$
\begin{tikzcd}
BU(i) \times BU(j) \arrow{d} \arrow{r} & BU(i) \times BU(j+1) \arrow{d} \\
BU(i+1) \times BU(j) \arrow{r} & BU(i+1) \times BU(j+1)
\end{tikzcd}
$$
do not commute on the nose, but only up to non-canonical homotopy.
\end{exm}

In Section \ref{sec:FilGra}, we discuss an \textit{associated graded} construction that transforms filtered $\mathbb{E}_n$-algebras into graded $\mathbb{E}_n$-algebras.  The central result of our paper may be thought of as a multiplicatively structured version of the Mitchell--Richter splitting:

\begin{thm} \label{thm:MainAoo}
As an $\mathbb{A}_\infty$-algebra object in filtered spectra, the Bott filtration of $\Sigma^{\infty}_+ \Omega SU(n)$ is equivalent to its associated graded.
\end{thm}

\begin{cor}
For any multiplicative homology theory $E$, $E_*(\Omega SU(n))$ is a bigraded ring.  One grading is given by $*$, and the other by the associated graded of the Bott filtration.
\end{cor}

\begin{exm} \label{exm:Thom}
In the case $n=\infty$ of Example \ref{exm:BottFil}, the Mitchell--Richter splitting
\begin{equation} \label{eqn:MunSplitting}
\Sigma^{\infty}_+ BU \simeq \bigvee_n \Sigma^{\infty} BU(n)/BU(n-1).
\end{equation}
recovers an older result of Snaith \cite{SnaithBook}.  Snaith further showed that (\ref{eqn:MunSplitting}) is an equivalence of homotopy commutative ring spectra, and our Theorem \ref{thm:MainAoo} gives an equivalence of $\mathbb{A}_\infty$-ring spectra.
 
The left-hand side of (\ref{eqn:MunSplitting}) is an $\mathbb{A}_\infty$-algebra by virtue of the fact that $BU$ is a loop space, while the right-hand side acquires its $\mathbb{A}_\infty$-algebra structure from an associated graded construction. To understand this latter $\mathbb{A}_\infty$ structure, it may help to know that $BU(n)/BU(n-1) \simeq MU(n),$ the Thom space of the canonical bundle over $BU(n)$.  We therefore recognize the right-hand side of (\ref{eqn:MunSplitting}) as the Thom spectrum of the $J$-homomorphism
$$\coprod BU(n) \stackrel{J}{\longrightarrow} Pic(\mathbb{S}).$$
Since $J$ is a loop map, its Thom spectrum acquires an $\mathbb{A}_\infty$-algebra structure, and this turns out to agree with the $\mathbb{A}_\infty$ associated graded of the Bott filtration.

Of course, $BU$ is not just a loop space, but in fact an infinite loop space.  Similarly, $J$ is not just a loop map, but furthermore an infinite loop map.  Thus, both sides of (\ref{eqn:MunSplitting}) are naturally $\mathbb{E}_\infty$-ring spectra.  Perhaps surprisingly, these $\mathbb{E}_\infty$-rings are \textbf{not} equivalent.
\end{exm}

\begin{rmk}
Snaith used his splitting $(\ref{eqn:MunSplitting})$ to give an equivalence of homotopy commutative ring spectra
\begin{equation}
\Sigma^{\infty}_+ BU[\beta^{-1}] \simeq MUP, \label{eqn:MUP}
\end{equation}
where $MUP$ denotes periodic complex bordism.  We will have more to say about the coherence of (\ref{eqn:MUP}) in forthcoming work.
\end{rmk}

In fact, the $\mathbb{A}_\infty$-splitting provided by Theorem \ref{thm:MainAoo} is the best result possible:

\begin{thm} \label{thm:MainObstruction}
Suppose $n \ge 4$.  If the Bott filtration of $\Sigma^{\infty}_+ \Omega SU(n)$ may be made into an $\mathbb{E}_2$-algebra object in filtered spectra, then it is \textbf{not} equivalent to its $\mathbb{E}_2$ associated graded.  More generally, any extension of the graded $\mathbb{A}_\infty$-algebra of Theorem \ref{thm:MainAoo} to a graded $\mathbb{E}_2$-algebra must fail to have the usual $\mathbb{E}_2$-algebra structure on its underlying ungraded $\mathbb{E}_2$-ring.
\end{thm}

However, if one is willing to work in a complex-oriented theory, such as ordinary homology, the situation improves:

\begin{thm} \label{thm:MainMUE2}
Let $MU$ denote the $\mathbb{E}_\infty$-ring spectrum of complex bordism.  Suppose that $R_1$ and $R_2$ are any two (ungraded) $\mathbb{E}_2$-algebras with the same underlying $\mathbb{A}_\infty$-ring $\Sigma^{\infty}_+ \Omega SU(n)$.  Then there there is an equivalence of $\mathbb{E}_2$-$MU$-algebras
$$MU \smsh R_1 \simeq MU \smsh R_2.$$
\end{thm}

For Theorem \ref{thm:MainMUE2} to be of interest, it is necessary to exhibit exotic $\mathbb{E}_2$-algebra structures on $\Sigma^{\infty}_+ \Omega SU(n)$.  Example \ref{exm:Thom} gives some idea of how this may be accomplished via Thom spectra, and we end Section \ref{sec:MRFil} with a sketch of the following:

\begin{cnstr} \label{cnstr:IntroGr}
There exists a graded $\mathbb{E}_2$-algebra structure on $\text{gr}(\Sigma^{\infty}_+ \{F_{n,k}\})$ extending the canonical graded $\mathbb{A}_\infty$-algebra structure.
\end{cnstr}

\begin{rmk}
Theorem \ref{thm:MainMUE2}, combined with Construction \ref{cnstr:IntroGr}, may be seen as a once-looped analogue of work of Kitchloo \cite{Kitchloo}.   Kitchloo studied a splitting, due to Miller \cite{MillerSplitting}, of $\Sigma^{\infty}_+ SU(n)$.  His theorem is that, \textit{for complex-oriented $E$}, the corresponding direct sum decomposition of $E_*(SU(n))$ is multiplicative.
\end{rmk}

Our proof of Theorem \ref{thm:MainMUE2} is by obstruction theory.  We show in Section \ref{sec:MUE2} that all obstructions to an $\mathbb{E}_2$ equivalence vanish.  On the other hand, we prove Theorem \ref{thm:MainObstruction} by explicitly calculating a non-zero obstruction in Section \ref{sec:Obstruction}.  It remains to discuss Theorem \ref{thm:MainAoo}, the $\mathbb{A}_\infty$ splitting of $\Sigma^{\infty}_+ \Omega SU(n)$, which is the central result of our paper.

To prove that a filtered spectrum
$$A_0 \longrightarrow A_1 \longrightarrow A_2 \longrightarrow A_3 \longrightarrow \cdots$$
splits, it suffices to provide splitting maps in the form of a ``cofiltered spectrum''
$$A_0 \longleftarrow A_1 \longleftarrow A_2 \longleftarrow A_3 \longleftarrow \cdots$$
In Section \ref{sec:FilGra} we make this statement precise (with the proof in Appendix \ref{app:SplittingMachine}) by explaining the following theorem:

\begin{customthm}{\ref{thm:SplitMachine}}
Let $X\in \Alg_{\E_n}(\Fil(\Sp))$ be an $\E_n$ filtered spectrum.  Suppose there exists an $\E_n$ cofiltered spectrum $Y\in \Alg_{\E_n}(\Cofil(\Sp))$ with the following two properties:
\begin{enumerate}
\item There is an equivalence $\mathrm{colim } X \simeq \lim Y$ of $\E_n$-algebras in spectra.
\item The resulting natural maps $X_i \to Y_i$ are equivalences.  
\end{enumerate}
Then, the filtered spectrum $X$ is $\E_n$-split.
\end{customthm}

We wish to apply this theorem in the case $n=1$, where $X$ is the $\mathbb{E}_1 = \mathbb{A}_\infty$ filtered spectrum in Theorem \ref{thm:BottIsAoo}.  In Section \ref{sec:MultWeiss}, we produce the corresponding $\mathbb{A}_\infty$ cofiltered spectrum; the proof is then finished in Section \ref{sec:AooSplit}.  To do this, we extend the methods of \cite{Arone}, who used Weiss calculus to give an elegant second proof of the Mitchell--Richter splitting (without multiplicative structure).  

Arone's idea is to use additional functoriality present in the filtration of $\Omega SU(n)$.  Let $\J$ denote the topological category of complex vector spaces and embeddings, fix a complex vector space $V$, and consider the functor $$G_V:\J \longrightarrow \mathcal{S}\text{paces}$$
given by $G_V(W)=\J(V,V \oplus W)$.  Observe that the special unitary group $SU(V \oplus \mathbb{C})$ arises as the value $G_V(\mathbb{C}) = \J (V,V\oplus \mathbb{C})$.  Weiss calculus provides a toolbox with which to study functors similar to $G_V$ -- a brief review of the theory is provided at the beginning of Section \ref{sec:MultWeiss}.  

The Bott filtration in fact arises from a sequence of functors $$F_0(W) \to F_1(W) \to F_2(W) \to \cdots \to F(W) := \Sigma^{\infty}_+ \Omega \J(V, V\oplus W)$$ from $\J$ to spectra.  This filtration has the key property that the successive quotients $F_n/F_{n-1}$ are \emph{homogeneous functors of degree n}.  Arone then applies an argument of Goodwillie \cite[Example 1.20]{GoodwillieIII} to see that in this situation, the Weiss polynomial approximation $P_n F(W)$ is precisely the functor $F_n(W)$.  The approximations $$P_n F(W) \rightarrow P_{n-1} P_n F(W) \simeq P_{n-1} F(W)$$ provide Arone with splitting maps.

To obtain an $\mathbb{A}_\infty$ splitting, it is no longer sufficient to merely provide splitting maps.  As explained by Theorem \ref{thm:SplitMachine}, we instead require an $\mathbb{A}_\infty$ structure on the whole system of splitting maps, considered as a cofiltered spectrum.  This will arise from combining two statements: the first is that the functor $F(W)$ takes values in $\mathbb{A}_\infty$ ring spectra, and thus gives an $\mathbb{A}_\infty$ object in the category $\Sp^{\J}$ of functors from $\J$ to spectra.  The second is that the Weiss tower can be viewed as a symmetric monoidal functor from $\Sp^{\J}$ to cofiltered objects in $\Sp^{\J}$.  This observation, which may be of independent interest, is likely known to experts and was suggested to us by Jacob Lurie, but we could not locate in the literature.  We have proven it in the following form, where $\Sp^{\J}_{\text{conv}}$ denotes restriction to certain conveniently convergent functors:
\begin{customthm}{\ref{thm:weissmonoidal}}
The Weiss tower defines a symmetric monoidal functor $$\text{Tow}: \Sp^{\J}_{\text{conv}} \to \Cofil(\Sp^{\J}_{\text{conv}}).$$
\end{customthm}
\begin{rmk*}
In fact, the convergence hypotheses were made for convenience and are not necessary.  Furthermore, the theorem works just as well in the context of Goodwillie calculus -- see Remark \ref{rmk:goodwilliecase}.  
\end{rmk*}

There are a number of natural and presumably approachable open questions suggested by our work here.  In addition to Question \ref{qst:BottE2}, we highlight the following:

\begin{qst}
Is the Mitchell--Richter filtration of the loop space of a Stiefel manifold always filtered $\mathbb{A}_\infty$?  This is the only obstruction to promoting Theorem \ref{thm:MainAoo} to a result about all such loop spaces.
\end{qst}

\begin{qst}
What are the proper motivic analogues of our results, phrased in the category of $\mathbb{A}^1$-invariant Nisnevich sheaves? 
\end{qst}

\begin{qst}
What are the proper equivariant analogues of our result?  See for example \cite{Ullman}, \cite{Tynan}.
\end{qst}

\subsection*{Acknowledgements}
The authors thank Greg Arone, Tom Bachmann, Lukas Brantner, Dennis Gaitsgory, Akhil Mathew, Haynes Miller, Denis Nardin, and Bill Richter for helpful conversations and the anonymous referee for numerous clarifying comments.  We were saddened to learn of the passing of Steve Mitchell during the preparation of this manuscript--his papers were a deep inspiration.  Special thanks are due to the authors' PhD advisors, Mike Hopkins and Jacob Lurie, both for their mathematical expertise and their consistent encouragement; many of the ideas in this paper grew out of their suggestions.  Special thanks are also due to Justin Campbell, James Tao, David Yang, and Yifei Zhao, all of whom spent numerous and invaluable hours answering naive questions about the Beilinson--Drinfeld Grassmannian.  The authors were supported by NSF Graduate Fellowships under Grants DGE-114415 and 1122374.

\subsection*{Notation and conventions:} 
\begin{enumerate}
\item As through the introduction, we will freely use the word category to refer to a not-necessarily truncated $\infty$-category and will specify when we explicitly work with $1$-categories.  
\item $\Sp$ denotes the category of spectra and $\mathcal{S}$ denotes the category of spaces.  
\item We will often conflate a graded spectrum with its underlying spectrum unless there is potential confusion.  
\item $\Z_{\geq 0}$ denotes the poset of non-negative integers as an ordinary category where $\Hom(a,b)$ is a singleton if $a\leq b$, and empty otherwise.  Denote by $\Z_{\geq 0}^{ds}$ the underlying set considered as a category with no morphisms.  We will implicitly take nerves to obtain $\infty$-categories which will serve as the indexing sets for filtered and graded spectra.  The reader is warned that our numbering conventions are opposite the ones in \cite{LurieRot}.
\end{enumerate}

\section{Filtered and Graded Ring Spectra} \label{sec:FilGra}

It will be important for us to have a precise language for discussing filtered and graded spectra, what it means to be split, what it means to take associated graded, and the multiplicative aspects of these constructions. Here we review a framework from \cite{LurieRot} for studying graded and filtered objects.  The reader is referred to \cite{LurieRot} for a more thorough treatment and all proofs.  

\subsection{First definitions}
Let $\D$ be an $\infty$-category which we will regard as the diagram category.  Our filtered objects will be valued in the functor category $\Sp^{\D}.$  This will be no more difficult than just ordinary spectra because limits, colimits, and smash products will be considered pointwise; in any case, we will refer to objects of $\Sp^{\D}$ as functors or simply as spectra.  

\begin{dfn} 
Let $\Gr(\Sp^{\D})$ denote the functor category $\Fun(\Z_{\geq 0}^{ds}, \Sp^{\D}).$  We shall refer to $\Gr(\Sp^{\D})$ as the category of graded objects in $\Sp^{\D}$.  Its objects can be thought of as sequences $X_0, X_1,X_2,\cdots \in \Sp^{\D}$.
\end{dfn}

\begin{dfn} 
Let $\Fil(\Sp^{\D})$ denote the functor category $\Fun(\Z_{\geq 0}, \Sp^{\D}).$  We shall refer to $\Fil(\Sp^{\D})$ as the category of filtered objects in $\Sp^{\D}$.  Its objects can be thought of as sequences $Y_0\to Y_1\to Y_2 \to \cdots \in \Sp^{\D}$ filtering $\colim_i Y_i$.  
\end{dfn}

\begin{rmk}\label{rmk:filtspaces}
We will occasionally also consider filtered objects in pointed spaces $\cS_*$, but we do not use them in an essential way in the paper so we restrict to spectra for the remainder of this section.  The issue is that the monoidal structure given by $\times$ does not preserve colimits separately in each variable -- for instance, it does not preserve the empty colimit.  The category $\Fil(\cS_*)$ of filtered pointed spaces nevertheless has a monoidal structure by Day convolution, but $\Gr(\cS_*)$ does not, and thus we do \emph{not} consider a monoidal structure on \emph{graded} pointed spaces.  
\end{rmk}

The obvious map $\Z_{\geq 0}^{ds} \to \Z_{\geq 0}$ induces a restriction functor $\text{res}: \Fil(\Sp^{\D}) \to \Gr(\Sp^{\D})$ which can be thought of as forgetting the maps in the filtered object.  The restriction fits into an adjunction  
$$I:\Gr(\Sp^{\D}) \xrightleftharpoons{\quad} \Fil(\Sp^{\D}) : \text{res}$$
where the left adjoint $I: \Gr(\Sp^{\D}) \to \Fil(\Sp^{\D})$ is given by left Kan extension.  The functor $I$ can be described explicitly as taking a graded object $X_0,X_1,X_2,\cdots$ to the filtered object $$I(X_0, X_1, \cdots) = (X_0\to X_0 \vee X_1\to X_0 \vee X_1 \vee X_2\to \cdots).$$   

Inverse to this, there is an associated graded functor $\text{gr }: \Fil(\Sp^{\D}) \to \Gr(\Sp^{\D})$ such that the composite $\text{gr }\circ I : \Gr(\Sp^{\D}) \to \Gr(\Sp^{\D})$ is an equivalence.   This can be thought of pointwise by the formula $$\text{gr}(X_0\to X_1\to X_2\to \cdots) = (X_0, X_1/X_0, X_2/X_1, \cdots).$$

As the names suggest, one may recover from a filtered or graded functor the underlying object.  For filtered objects, this is a functor $$\colim : \Fil(\Sp^{\D}) \to \Sp^{\D}$$ given by Kan extending along $\Z_{\geq 0}\to *.$   It can be thought of as taking the colimit.  For graded objects, the underlying object is simply the direct sum of all the graded pieces. We will systematically abuse notation by conflating a graded spectrum with its underlying spectrum when we feel there is no potential for ambiguity.


\subsection{Monoidal structures}\label{sect:monoidal}
We now begin studying the monoidal structures on graded and filtered spectra.  We confine ourselves to a basic discussion here, leaving a more technical discussion for Appendix \ref{app:day}.

By \cite{Glasman} or \cite[Example 2.2.6.17]{HA}, the categories $\Gr(\Sp)$ and $\Fil(\Sp)$ may be given symmetric monoidal structures via the Day convolution.  Then, via the identifications $\Gr(\Sp^{\D}) = \Gr(\Sp)^{\D}$ and $\Fil(\Sp^{\D}) = \Fil(\Sp)^{\D}$, the categories $\Gr(\Sp^{\D})$ and $\Fil(\Sp^{\D})$ may be given symmetric monoidal structures pointwise on $\D$.  In both cases, we denote the resulting operation by $\otimes$. Explicitly, the filtered tensor product $$\left(X_0 \longrightarrow X_1 \longrightarrow X_2 \longrightarrow \cdots \right) \otimes \left(Y_0 \longrightarrow Y_1 \longrightarrow Y_2 \longrightarrow \cdots \right)$$
of two filtered spectra is computed as

\begin{center}
$X_0 \smsh Y_0 \longrightarrow \colim $
\adjustbox{scale=0.7} 
{$ \left(\begin{tikzcd} X_0 \smsh Y_1 \\  X_0 \smsh Y_0 \arrow{u} \arrow{r} & X_1 \smsh Y_0 \end{tikzcd} \right) $} 
$\longrightarrow \colim$
\adjustbox{scale=0.7} {$ \left( \begin{tikzcd} X_0 \smsh Y_2 \\ X_0 \smsh Y_1 \arrow{r} \arrow{u} & X_1 \smsh Y_1  \\ X_0 \smsh Y_0 \arrow{r} \arrow{u} & X_1 \smsh Y_0 \arrow{u} \arrow{r} & X_2 \smsh Y_0 \end{tikzcd} \right) $}
$\longrightarrow \cdots.$
\end{center}

For graded spectra, the analogous formula is:

$$(A_0,A_1,A_2,\cdots) \otimes (B_0,B_1,B_2,\cdots) \simeq \left( A_0 \smsh B_0, (A_1 \smsh B_0) \vee (A_0 \smsh B_1), \cdots, \bigvee_{i+j=n} A_i \smsh B_j, \cdots \right).$$

The unit $\mathbb{S}^{gr}_{\D}$ of $\otimes$ in $\Gr(\Sp^{\D})$ is the constant diagram at $S^0$ in degree 0 and $*$ otherwise; the unit $\mathbb{S}^{fil}_{\D}$ in $\Fil(\Sp^{\D})$ is $I\mathbb{S}^{gr}_{\D}.$  We may then talk about $\E_n$-algebras in $\Gr(\Sp^{\D})$ and $\Fil(\Sp^{\D})$ and their modules.

\begin{rmk}
One consequence of this is that one can recover filtered spectra as a module category inside the category of graded spectra.  There is a symmetric monoidal functor $\mathbb{Z}^{ds}_{\ge 0} \to \Sp$ sending each integer to $S^0$; this yields an $\mathbb{E}_\infty$-algebra $A=\Sigma^{\infty}_+ \mathbb{Z}^{ds}_{\ge 0}$ in $\Gr(\Sp)$.  The spectrum underlying $A$ is an infinite wedge of copies of $\mathbb{S}^0$.  Given a filtered spectrum $X$, $\mathrm{res }(X)$ acquires an action of $A$ where the $S^0$ in degree $1$ acts by ``shifting filtration.''  In fact, we have the following, as proven in \cite[Proposition 3.1.6]{LurieRot}:

\begin{lem} \label{lem:FilAsGrMod}
The functor $\mathrm{res}$ lifts to an equivalence
$$\Fil(\Sp) \stackrel{\simeq}{\longrightarrow} \mathbf{Mod}_{A}(\Gr(\Sp))$$
of symmetric monoidal $\infty$-categories.
\end{lem}
\end{rmk}

The functors $I$ and $\text{gr}$ can be given symmetric monoidal structures such that the composite $\text{gr}\circ I : \Gr(\Sp^{\D}) \to \Gr(\Sp^{\D})$ is a symmetric monoidal equivalence by \cite[Proposition 3.2.1]{LurieRot}.  It follows in particular that they extend to functors between the categories of $\E_n$-algebras in $\Gr(\Sp^{\D})$ and $\Fil(\Sp^{\D})$.  Thus, given an $\E_n$-algebra $Y$ in filtered spectra, we obtain a canonical $\E_n$ structure on its associated graded $\text{gr}(Y).$  Conversely, given $X\in \Alg_{\E_n}(\Gr(\Sp^{\D}))$, we obtain $IX\in \Alg_{\E_n}(\Fil(\Sp^{\D})).$  

\subsection{Structured splittings}
In this paper, we aim to study filtered spectra which are split in a way compatible with multiplicative structure.  This is captured by the following definition:

\begin{dfn}
A filtered $\mathbb{E}_n$-algebra $X\in \Alg_{\E_n}(\Fil(\Sp^{\D}))$ is called \emph{$\E_n$-split} if there exists some $Y \in \Alg_{\E_n}(\Gr(\Sp^{\D}))$ and an equivalence $X \simeq IY$ in $\Alg_{\E_n}(\Fil(\Sp^{\D}))$.  
\end{dfn}

Given an $\E_n$-split filtered spectrum $X$, we can recover the underlying graded spectrum by taking the associated graded.  

\begin{exm}\label{exm:snaith}
In this example, we relate the Snaith splitting to the above notions of $\E_n$-split filtered spectrum.  Since the functors $\mathrm{res}$ and $\Omega^{\infty}$ are lax monoidal, we may consider the commutative diagram of right adjoints
$$
\begin{tikzcd}
\Gr(\Sp) & \Fil(\Sp) \arrow[l,"\mathrm{res}"] \arrow[r,"\Omega^{\infty}"] & \Fil(\cS_*) \\
\Alg_{\E_n}(\Gr(\Sp)) \arrow[u] & \Alg_{\E_n}(\Fil(\Sp))  \arrow[l,"\mathrm{res}"] \arrow[r,"\Omega^{\infty}"] \arrow[u]& \Alg_{\E_n}(\Fil(\cS_*)) \arrow[u]
\end{tikzcd}
$$
where the vertical maps forget the algebra structure.  This induces a corresponding commutative diagram of left adjoints
$$
\begin{tikzcd}
\Gr(\Sp) \arrow[r,"I"] \arrow[d,"F_{\E_n}"] & \Fil(\Sp)\arrow[d,"F_{\E_n}"] & \Fil(\cS_*) \arrow[d,"F_{\E_n}"]\arrow[l,"\Sigma^{\infty}"]\\
\Alg_{\E_n}(\Gr(\Sp))\arrow[r,"I"] & \Alg_{\E_n}(\Fil(\Sp))  & \Alg_{\E_n}(\Fil(\cS_*)) \arrow[l,"\Sigma^{\infty}_+"]
\end{tikzcd}
$$
where the vertical maps take free algebras.  Let $E$ be a spectrum and consider $E[1]$, the graded spectrum with $E$ in degree $1$.  The commutativity of the left square shows that the free $\E_n$-algebra in filtered spectra on $I(E[1])$ is $\E_n$-split.  Moreover, let $X\in \cS_*$ be a pointed space and $n>0$; work of \cite{May} determines the free $\E_n$-algebra on $X$ in pointed spaces to be $\Omega^n \Sigma^n X$, which comes with a canonical filtration.  We claim without proof that this is the free $\E_n$-filtered space on $X$ placed in filtration $1$.  The right square in this diagram verifies that its suspension is the free $\E_n$-algebra in filtered spectra on $I(\Sigma^{\infty} X [1])$, and thus, by the previous remarks, yields an $\E_n$-split filtered spectrum, demonstrating a structured version of the Snaith splitting \cite{SnaithSplit}.
\end{exm}

In this paper, we will be interested in when a given $\E_n$ filtered spectrum is $\E_n$-split.  Disregarding the multiplicative structure, a filtered spectrum $$X_0\longrightarrow X_1 \longrightarrow X_2 \longrightarrow \cdots ,$$ will split if and only if there are maps going the other way: $$X_0 \longleftarrow X_1 \longleftarrow X_2 \longleftarrow \cdots,$$ with the property that the relevant composites are equivalences.   To systematically talk about these backwards maps, we need the following definition:

\begin{dfn} Let $\Cofil(\Sp^{\D})$ denote the functor category $\Fun(\Z_{\geq 0}^{op}, \Sp^{\D}).$  We shall refer to $\Cofil(\Sp^{\D})$ as the category of cofiltered objects in $\Sp^{\D}$.  Its objects can be thought of as towers of functors $Y_0\leftarrow Y_1\leftarrow Y_2 \leftarrow \cdots \in \Sp^{\D}$.
\end{dfn}

In Appendix \ref{app:day}, we show how to give $\Cofil(\Sp^{\D})$ the structure of a symmetric monoidal $\infty$-category.  One might then correctly surmise that producing a multiplicatively structured splitting involves the multiplicativity of this opposite filtration.  In particular, we have the following criterion, which we prove in Appendix \ref{app:SplittingMachine} (we now switch from $\Sp^{\D}$ to $\Sp$ for ease of notation):

\begin{thm}\label{thm:SplitMachine}
Let $X\in \Alg_{\E_n}(\Fil(\Sp))$ be an $\E_n$ filtered spectrum.  Suppose there exists an $\E_n$ cofiltered spectrum $Y\in \Alg_{\E_n}(\Cofil(\Sp))$ with the following two properties:
\begin{enumerate}
\item There is an equivalence $\mathrm{colim } X \simeq \lim Y$ of $\E_n$-algebras in spectra.
\item The resulting natural maps $X_i \to Y_i$ are equivalences.  
\end{enumerate}
Then, the filtered spectrum $X$ is $\E_n$-split.
\end{thm}


\section{The \texorpdfstring{$\mathbb{E}_2$}{E2} Schubert Filtration} \label{sec:Schubert}

In the next section we will define the Mitchell--Richter Bott filtration on $\Omega SU(n)$ and prove that it makes $\Sigma^{\infty}_+ \Omega SU(n)$ into a filtered $\mathbb{A}_\infty$-ring spectrum.  The filtration, especially as a multiplicative object, is most naturally described in the language of algebraic geometry.  Our aim in this section will be to recall the complex points of these algebro-geometric objects and extract from both the algebro-geometric and topological literature all of the basic facts about these complex points that we will later need.  None of the ideas in this section are original: the objects we study are due to Beilinson and Drinfeld \cite{BDQuantization}, and their translation into algebraic topology is due to Jacob Lurie  \cite[\S 5.5]{HA}.

\subsection{The Schubert filtration on the Affine Grassmannian}

Fix for the moment a smooth, reductive, affine algebraic group $G$ over $\mathbb{C}$.  Through the rest of the paper, we will be interested only in the cases $G=\mathbb{G}_m,SL_n,$ or $GL_n$, and so the reader may safely restrict their attention to those cases for concreteness.  As we will explain in detail, an algebro-geometric model for $\Omega G(\mathbb{C})$ is the \textit{affine Grassmannian} $Gr_G$.  The $\mathbb{E}_2$-algebra structure present on $\Omega G(\mathbb{C})$ is encoded by a more elaborate object, the \textit{Beilinson--Drinfeld Grassmannian}.  A good general reference for both of these objects is \cite{Zhu}, whose presentation we will more or less follow below.

We use $D$ to denote the formal disk $\text{Spec}(\mathbb{C}[[t]])$ and $D^*$ to denote the punctured disk $\text{Spec}(\mathbb{C}((t)))$.  For $R$ a $\mathbb{C}$-algebra, we use $D_R$ to denote $\text{Spec}(\mathbb{C}[[t]] \hat{\otimes} R)$ and $D^*_R$ to denote $\text{Spec}(\mathbb{C}((t)) \hat{\otimes} R)$.  We refer to a point in the space $G(D^*)=G(\mathbb{C}((t)))$ as an algebraic free (i.e., unbased) loop in $G$.  Such points are exactly automorphisms of the trivial $G$-torsor $\mathcal{E}^0$ over $D^*$.

\begin{dfn}
The \textit{affine Grassmannian} $Gr_G$ of $G$ is the Ind-scheme with functor of points
$$Gr_G(R) = \{(\mathcal{E},\beta)\},\text{ where}$$
$\mathcal{E}$ is a $G$-torsor over $D_{R}$ and $\beta:\mathcal{E}|_{D^*_{R}} \cong \mathcal{E}^0_{D^*_{R}}$ is a trivialization over $D^*_{R}$.
\end{dfn}

The complex points $Gr_G(\mathbb{C})$ are a model for the topological space $\Omega G$ \cite[\S 8.3]{PressleySegal}, \cite[1.6]{Zhu}.  These complex points are given \cite[1.4]{Zhu} by the homogeneous space
$$G(\mathbb{C}((t)))/G(\mathbb{C}[[t]]),$$
which up to homotopy is the quotient of the free loop space on $G$ by the action of $G$.

The above functor of points is not representable by a scheme, but it is a filtered colimit of schemes $Gr_{G,\le \mu}$.  We will define below at least the complex points $Gr_{G, \le \mu}(\mathbb{C})$ of these schemes, which are specific topological subspaces of $Gr_G(\mathbb{C})$.  First, it is necessary to introduce a bit more notation.

Let $T$ denote a maximal torus inside $G$.  We use $\mathbb{X}^{\bullet}$ to denote the lattice of weights $\Hom(T,\mathbb{G}_m)$, and $\mathbb{X}_{\bullet}$ to denote the dual lattice of coweights.  Inside $\mathbb{X}^{\bullet}$ is the set $\Phi$ of roots.  We fix a particular Borel subgroup $B \subset G$, determining a choice of positive roots $\Phi^+ \subset \Phi$ and a semi-group of dominant coweights $\mathbb{X}^+_\bullet \subset \mathbb{X}_\bullet$.  There is a natural bijection
$$\mathbb{X}_{\bullet}^+ \cong G(\mathbb{C}[[t]])\backslash G(\mathbb{C}((t)))/G(\mathbb{C}[[t]])$$
of dominant coweights with the above double cosets.  Each coweight $\mu \in \Hom(\mathbb{G}_m,T)$ defines via the inclusion of $T$ into $G$ an element $t^{\mu}$ in $G(\mathbb{C}((t)))$.  There is a Bruhat decomposition of the algebraic free loop space
$$
G(\mathbb{C}((t))) \cong \coprod_{\mu \in \mathbb{X}_{\bullet}^+} G(\mathbb{C}[[t]]) t^{\mu} G(\mathbb{C}[[t]]).
$$

Projecting onto the affine Grassmannian, one learns that the $G(\mathbb{C}[[t]])$-orbits of $Gr_G$ are indexed by $\mu \in \mathbb{X}_{\bullet}^+$.
We will use $Gr_{G,\le \mu}$ to denote the \textit{closure} of the orbit corresponding to $\mu$.  The closure $Gr_{G, \le \mu_1}$ contains $Gr_{G, \le \mu_2}$ if and only if $\mu_1-\mu_2$ is a sum of dominant coroots \cite[2.1]{Zhu}.  This ordering on the dominant coweights is known as the Bruhat order.  

\begin{dfn} \label{def:coweightfiltered}
An $\mathbb{X}_{\bullet}^+$-filtered topological space $K$ is a poset of topological subspaces of $K$, indexed by $\mathbb{X}_{\bullet}^+$.
\end{dfn}

An example is of course given by $Gr_{G}(\mathbb{C})$, filtered by the $Gr_{G,\le \mu}(\mathbb{C})$.  We will refer to this filtration on $Gr_{G}(\mathbb{C})$ as the \textit{Schubert filtration}.

\begin{rmk} \label{rem:monodromy}
Suppose that $\gamma$ is a principal $G$-bundle on $\mathbb{A}^1$, and that $p$ is the origin in $\mathbb{A}^1$.  Then restriction of $\gamma$ to an infinitesimal neighborhood of $p$ gives a principal $G$-bundle on the formal disk $D$.  If one is then further given a trivialization of $\gamma$ away from $p$, there is then an associated point $x \in Gr_{G}(\mathbb{C})$.  This point $x$ lives in a certain $G(\mathbb{C}[[t]])$-orbit, and thus has well-defined \textit{monodromy} $\mu \in \mathbb{X}_{\bullet}^+$.

Now, if $p$ is \textit{any} point in $\mathbb{A}^1$, then a formal neighborhood of $p$ is isomorphic but not equal to the formal disk $D$.  If given a trivialization of $\gamma$ away from $p$, this means that one does not get a well-defined point $x \in Gr_{G}(\mathbb{C})$, but one \textit{does} obtain a well-defined $G(\mathbb{C}[[t]])$-orbit inside of $Gr_{G}(\mathbb{C})$.  In other words, the trivialization of $\gamma$ away from $p$ has well-defined monodromy $\mu \in \mathbb{X}_{\bullet}^+$.  Notice, in fact, that to define $\mu$ one needs not a trivialization on all of $\mathbb{A}^1 \backslash \{ p \}$, but only a trivialization defined away from $p$ in a neighborhood of $p$.
\end{rmk}

\begin{exm} \label{sl2example}
Suppose $G=SL_2(\mathbb{C})$ with its usual maximal torus.  A coweight $\mu \in \mathbb{X}_\bullet$ consists of a pair $(a,b)$ of integers with $a+b=0$.  We choose a Borel so that a coweight is dominant if $a \ge b$.  The conjugation action of $SL_2(\mathbb{C})$ on $\Omega SL_2(\mathbb{C})$ has one orbit for each pair $(a,-a)$ with $a \ge 0$.  The orbit corresponding to $(a,-a)$ contains the loop $\mathbb{G}_m \rightarrow \Omega SL_2(\mathbb{C})$ given by
$$
t \mapsto \left( \begin{array}{cc} t^a & 0 \\ 0 & t^{-a}  \end{array} \right).
$$
The closure of the $(a,-a)$ orbit contains the $(b,-b)$ orbit if and only if $b \le a$.  To topologists, $\Omega SL_2(\mathbb{C}) \simeq \Omega \Sigma S^2$ is recognizable as the free $\mathbb{A}_\infty$-algebra on the pointed space $S^2$.  In particular, $Gr_{SL_2}(\mathbb{C})$ is naturally equipped with the James filtration by word length.  The closure of the $(a,-a)$ orbit, denoted $Gr_{SL_2,\le (a,-a)}(\mathbb{C})$, turns out to be the $(2a)$th component of the James filtration.

Thus, the Schubert filtration is strictly coarser than the James filtration.  In particular, the $S^2$ that appears as the first James filtered piece of $\Omega SL_2(\mathbb{C})$ is not closed under the $SL_2(\mathbb{C})$ conjugation action.  Only the collection of words of length $2$ or less is closed under the $SL_2(\mathbb{C})$ action.

In this paper, we are interested in a filtration called the \emph{Bott filtration} of $\Omega SL_n(\mathbb{C}).$  For $\Omega SL_2(\mathbb{C}),$ the Bott filtration corresponds to the James filtration on $\Omega S^3$, and in particular is \textbf{not} given by the Schubert filtration on $Gr_{SL_2}(\mathbb{C})$.  However, in Section \ref{sec:MRFil}, we will explain how to obtain the Bott filtration from the Schubert filtration, while in the remainder of this section we will explain how the Schubert filtration interacts with the $\mathbb{E}_2$-algebra structure on $\Omega SL_2(\mathbb{C})$.
\end{exm}

\subsection{$\mathbb{E}_2$-algebras via disk algebras}

Let $\C$ be a symmetric monoidal category.  An $\mathbb{E}_2$-algebra in $\C$ is an object $A\in \C$ with multiplications parametrized by embeddings of disjoint unions of disks.  Our preferred model for this will be the notion of a $N(\mathrm{Disk}(\mathbb{C}))$-algebra as in \cite[\S 5.4.5]{HA}.  There is a colored operad $\mathrm{Disk}(\mathbb{C})$ whose colors are disks (open subspaces of $\mathbb{C}$ which are homeomorphic to $\mathbb{R}^2$).  The set of operations from $(D_1, D_2, \cdots, D_n)$ to $D$ is the singleton if the $D_i$ are disjoint and contained in $D$, and empty otherwise.  From this colored operad, we may obtain an $\infty$-operad $\N(\mathrm{Disk}(\mathbb{C}))^{\otimes}$ whose algebras we will refer to as $\N(\mathrm{Disk}(\mathbb{C}))$-algebras.  A $\N(\mathrm{Disk}(\mathbb{C}))$ algebra $A$ is the data of an object $A(D)\in\C$ for each disk $D\subset \mathbb{C}$ and coherent maps $A(D_1)\otimes A(D_2)\otimes \cdots \otimes A(D_n) \to A(D)$ for inclusions $D_1\coprod \cdots \coprod D_n \to D$ of disks. 

The following theorem relates these to (non-unital) $\mathbb{E}_2$-algebras\footnote{For technical reasons, we will work with the non-unital (abbreviated `$\mathrm{nu}$') variants of $\mathbb{E}_2$ and $\mathrm{Disk}(\mathbb{C})$-algebras from now on.}:

\begin{thm}[Proposition 5.4.5.15, \cite{HA}]\label{thm:lcdisk}
There is a fully faithful functor $\Alg^{\mathrm{nu}}_{\mathbb{E}_2}(\C) \to \Alg^{\mathrm{nu}}_{\N(\mathrm{Disk}(\mathbb{C}))}(\C)$ with essential image the full subcategory of $\N(\mathrm{Disk}(\mathbb{C}))_{\mathrm{nu}}$-algebras $A$ which are \emph{locally constant} in the sense that for every embedding $D\hookrightarrow D'$ of disks, the natural map $A(D)\to A(D')$ is an equivalence.  
\end{thm}

In this paper, we will construct $\mathbb{E}_2$-algebras by constructing $\N(\mathrm{Disk}(\mathbb{C}))_{\mathrm{nu}}$-algebras and then checking the condition of Theorem \ref{thm:lcdisk}.  We will be particularly interested in equipping the affine Grassmannian with the structure of an $\mathbb{E}_2$-algebra in filtered spaces.

Before returning to the setting of the affine Grassmannian, it will be helpful to record here one additional topological construction.

\begin{dfn}[\S 3.4, \cite{BD}]
For any topological space $X$, the collection of nonempty finite subsets of $X$ can be itself topologized as follows: consider the contravariant functor from the category $\Fin_{\text{surj}}$ of finite sets and surjections to the category of topological spaces carrying a finite set $S$ to $X^S$ and a surjection $S \twoheadrightarrow T$ to the natural diagonal map $X^T \to X^S.$
Let $\Fin^{\leq n}_{\mathrm{surj}}$ denote the category of finite sets of cardinality at most $n$.  We define the topological spaces $$\Ran^{\leq n}(X) := \colim_{S\in \Fin^{\leq n}_{\mathrm{surj}}} X^S$$ and $$\Ran(X) := \colim_{n\to \infty} \Ran^{\leq n}(X)$$ where the colimits are taken in the 1-category of topological spaces.  We will refer to $\Ran(X)$ as the \emph{Ran space} of $X$.  
\end{dfn}

\begin{rmk}
There is another topology on the Ran space, given for instance in \cite[\S 5.5]{HA}.  For our purposes, it will not matter which topology we use, so we will use the colimit topology throughout.  
\end{rmk}

\begin{rmk}\label{rmk:factcosh}
Morally, the Ran space gives yet another way to think about $\mathbb{E}_2$-algebras.  Indeed, an $\E_2$ algebra $A$ can be thought of as consisting of a copy for $A$ of each disk in the plane together with multiplications parametrized by embeddings of disjoint unions of disks.  On an infinitesimal scale, this can be thought of as giving for each point in the plane $x\in \mathbb{C}$ a copy $A_x$ of $A$.  The embeddings of disks correspond to collisions of points in the plane.  The Ran space is built exactly to track this collision data.  More precisely, given an $\mathbb{E}_2$-algebra $A$ and a finite set of points in the plane $\{x_i\}_{i\in I} \in \Ran(\mathbb{C})$, one can form the tensor product $\bigotimes_{x\in X} A_x.$  These tensor products turn out to be stalks of a cosheaf on $\Ran(\mathbb{C})$; the functoriality in the Ran space expresses the idea of points colliding determining the multiplicative structure.

The more precise connection between the Ran space and disk algebras is given by the notion of a \emph{factorizable cosheaf} on the Ran space of $\mathbb{C}$ \cite[Theorem 5.5.4.10]{HA}.  Instead of introducing the general theory, we will turn immediately to our example of interest, the Beilinson-Drinfeld Grassmannian.  
\end{rmk}

\subsection{The Beilinson-Drinfeld Grassmannian}
The structure alluded to in the previous section arises naturally in algebraic geometry in the form of the Beilinson-Drinfeld Grassmannian.  Our goal in this section is to set up this object and describe the $\mathbb{E}_2$ algebras it produces via Theorem \ref{thm:lcdisk}.

\begin{dfn}
The \textit{algebro-geometric Ran space} $\text{Ran}_{\mathbb{A}^1}$ is the presheaf that assigns to every $\mathbb{C}$-algebra $R$ the set of non-empty finite subsets of $\text{Spec}(R) \times \mathbb{A}^1$ over $\text{Spec}(R)$.   The Beilinson--Drinfeld Grassmannian is the presheaf $Gr_{G,\text{Ran}}$ that assigns to each $\mathbb{C}$-algebra $R$ the set of triplets $(x,\mathcal{E},\beta)$, where $x \in \text{Ran}_{\mathbb{A}^1}(R)$, $\mathcal{E}$ is a $G$-torsor on $\mathbb{A}^1 \times \text{Spec}(R)$, and $\beta$ is a trivialization of $\mathcal{E}$ away from the graph of $x$ in $\text{Spec}(R) \times \mathbb{A}^1$. 
\end{dfn}

The Beilinson--Drinfeld Grassmannian is naturally fibered over the algebro-geometric Ran space.  Note that the complex points of the algebro-geometric Ran space are just the space $\Ran (\mathbb{C})$ that we have previously considered.  We shall be primarily interested in the resulting map on complex points $$\Gra_{G,\Ran}(\mathbb{C}) \xrightarrow{p} \Ran(\mathbb{C}).$$ For a subset $U\subset \mathbb{C}$, let $\Gra_{G,\Ran}(U\subset \mathbb{C}):= p^{-1}(\Ran(U)).$   As explained above, a point $x$ in $\text{Ran}(\mathbb{C})$ consists of a non-empty finite subset $I \subset \mathbb{C}$ of points in $\mathbb{C}$.  The fiber of the Beilinson--Drinfeld Grassmannian $Gr_{G,\text{Ran}}(\mathbb{C})$ over $x$ is the moduli of $G$-bundles on $\mathbb{A}^1$ equipped with a trivialization away from the points in $I$.  This fiber is non-canonically isomorphic to the product of $|I|$ copies of the affine Grassmannian $Gr_G(\mathbb{C})$.  

Via the machinery in the previous section (as we will explain in more detail later), $Gr_{G,\text{Ran}}(\mathbb{C})$ describes the $\mathbb{E}_2$-algebra structure on the affine Grassmannian.  However, we are interested in the compatibility of this $\mathbb{E}_2$ structure with certain filtrations on $Gr_G(\mathbb{C})$.  We must therefore exhibit a filtration on the whole Beilinson-Drinfeld Grassmannian.  In particular, we filter $Gr_{G,\text{Ran}}(\mathbb{C})$ by topological spaces $Gr_{G,\text{Ran},\le \mu}(\mathbb{C})$, allowing us to view $Gr_{G,\text{Ran}}(\mathbb{C})$ as an $\mathbb{X}_{\bullet}^+$-filtered topological space in the sense of Definition \ref{def:coweightfiltered}.  

To do so, consider an arbitrary point $p \in \text{Gr}_{G,\text{Ran}}(\mathbb{C})$, consisting of a $G$-bundle $\gamma$ on $\mathbb{A}^1$ trivialized away from a collection of points $I=(i_1,i_2,...,i_\ell) \subset \mathbb{C}$.  For each point $i_k \in \mathbb{C}$, the restriction of $\gamma$ to a formal neighborhood of $i_k$ has well-defined monodromy $\mu_k$, in the sense of Remark \ref{rem:monodromy}.  We say that $p$ lives in $\text{Gr}_{G,\text{Ran},\le \mu}(\mathbb{C})$ if and only if
$$\sum_{i=1}^{k} \mu_k \le \mu.$$  This defines $\Gra_{G,\Ran}(\mathbb{C})$ as an $\mathbb{X}^+_{\bullet}$-filtered topological space (for a filtration at the level of algebro-geometric objects, see \cite[3.1.11]{Zhu}).

We will ultimately be concerned with certain graded and filtered spectra that one can extract from $\Gra_{G}(\mathbb{C})$ and its coweight filtration.  In particular, given a dominant coweight $\mu \in \mathbb{X}^+_{\bullet}$, one obtains a graded spectrum $ \{ \Sigma^{\infty}_+ \Gra_{G,\leq k\mu}(\mathbb{C})\}_{k\in \mathbb{Z}_{\geq 0}}$.  Let us now choose and fix a dominant coweight $\mu$. Our goal for the remainder of the section is to prove the following fact:

\begin{thm}\label{thm:schubertE2}
Let $G$ be a smooth, reductive affine algebraic group over $\mathbb{C}$ and $\mu\in \mathbb{X}^+_{\bullet}$ be a dominant coweight.  Then, the graded spectrum $\{\Sigma^{\infty}_+ \Gra_{G, \leq k \mu}(\mathbb{C}) \}_{k\in \mathbb{Z}_{\geq 0}}$ naturally admits the structure of an $\mathbb{E}_2$-algebra in graded spectra. 
\end{thm}

\begin{rmk}\label{rmk:choosespectra}
We have chosen to state our theorem in spectra because we do not set up the monoidal structure on the category of graded spaces (cf. Remark \ref{rmk:filtspaces}).  One could make sense of $\{ \Gra_{G,\leq k \mu}(\mathbb{C}) \}_{k\in \mathbb{Z}_{\geq 0}}$ as a \emph{non-unital} $\mathbb{E}_2$-algebra in graded spectra, but this is unnecessary for our applications.  Nevertheless, the proof of Theorem \ref{thm:schubertE2} goes through in ungraded spaces to show that $\coprod_{k\in \mathbb{Z}_{\geq 0}} \Gra_{G,\leq  k\mu}(\mathbb{C})$ is an $\E_2$-algebra in pointed spaces.  
\end{rmk}

\begin{rmk}
It is not generally true that $\Gra_{G,\leq k\mu}$ is contained in $\Gra_{G, \leq j\mu}$ when $k<j$ because this does not imply that $k\mu \leq j\mu$ in the poset of coweights.  However, in the case when $\mu \geq 0$, then $k\mu < j\mu$ and the pieces of the graded spectrum $ \{ \Sigma^{\infty}_+ \Gra_{G,\leq k\mu}(\mathbb{C})\}_{k\in \mathbb{Z}_{\geq 0}}$ are naturally included in one another.  In this situation, one obtains a filtered spectrum which can similarly be shown to be an $\mathbb{E}_2$-algebra in filtered spectra.  
\end{rmk}

\begin{rmk} \label{rmk:bdgrfunct}

The proof of Theorem \ref{thm:schubertE2} will also make clear the functoriality of the construction.  In particular, suppose that $G_1 \longrightarrow G_2$ is a map of reductive groups, and further suppose that this map is compatible with choices of maximal tori and dominant coweights, so that it induces a map of dominant coweight lattices $\mathbb{X}_{1,\bullet}^+ \longrightarrow \mathbb{X}_{2,\bullet}^+$.  Then if a dominant coweight $\mu_1$ is sent to a dominant coweight $\mu_2$, there is an induced map of graded $\mathbb{E}_2$-algebras
$$\{\Sigma^{\infty}_+ \Gra_{G_1, \leq k \mu_1}(\mathbb{C}) \}_{k\in \mathbb{Z}_{\geq 0}} \longrightarrow \{\Sigma^{\infty}_+ \Gra_{G_2, \leq k \mu_2} (\mathbb{C}) \}_{k\in \mathbb{Z}_{\geq 0}}.$$
\end{rmk}

\begin{dfn}
The 1-category of graded topological spaces is the ordinary category of functors from $\mathbb{Z}_{\geq 0}^{\mathrm{ds}}$ to topological spaces.  It acquires the structure of a symmetric monoidal category by the Cartesian product $\times$ with the usual grading conventions.  
\end{dfn}

\begin{cnstr}\label{cnstr:diskalg}
We define a $\mathrm{Disk}(\mathbb{C})_{\mathrm{nu}}$ algebra $\bar{A}$ valued in graded topological spaces.  This will arise from the map $p:\Gra_{G,\Ran}(\mathbb{C}) \to \Ran(\mathbb{C})$ constructed above.  For a disk $D\subset \mathbb{C}$, we define the graded topological space $\bar{A}(D)$ by the formula $\bar{A}(D)_k := \Gra_{G,\Ran, \leq k\mu}(D\subset \mathbb{C}).$  Explicitly, the points of $\bar{A}(D)_k$ are given by the data of a pair $(I, \mathcal{F})$ where $I\subset D$ is a nonempty finite subset and $\mathcal{F}$ is a principal $G$-bundle with a trivialization away from $I$ with monodromy less than or equal to $k\mu$ in the sense defined above.  Let $D_1 \coprod D_2\coprod \cdots \coprod D_n \hookrightarrow D$ be a nonempty inclusion of disjoint disks into a larger disk $D$.  To define a $\mathrm{Disk}(\mathbb{C})_{\mathrm{nu}}$ algebra, we need to exhibit maps $\bar{A}(D_1)\times \cdots \times \bar{A}(D_n) \to \bar{A}(D)$ satisfying the obvious coherences.  
This map sends the point determined by $(I_1, \mathcal{F}_1), \cdots, (I_n, \mathcal{F}_n)$ to the point $(I, \mathcal{F})$ where $I= I_1 \coprod \cdots \coprod I_n$ and $\mathcal{F}$ is obtained by gluing the $\mathcal{F}_i$ along the trivializations away from $I$.  This defines a map respecting the grading because if the coweight corresponding to $(I_i,\mathcal{F}_i) \in \bar{A}(D)_{k_i}$ is $\alpha_i\leq k_i \mu$, then $\Sigma_i \alpha_i \leq (\Sigma_i k_i) \mu$ and so one naturally lands in grading $\Sigma_i k_i$.  
\end{cnstr}
\begin{rmk}\label{rmk:bdgrfact}
More generally, the Beilinson-Drinfeld Grassmannian yields the structure of a factorizable cosheaf on $\Ran(\mathbb{C})$ as mentioned in Remark \ref{rmk:factcosh}.  The key structure present is that of \emph{factorization}, which roughly says that the Grassmannian over the disjoint union of two subsets $U,V\subset \mathbb{C}$ is equivalent to the product of the Grassmannians over the two sets separately. 
\end{rmk}

Passing to the $\infty$-category of spaces and stabilizing, we obtain a $\N(\mathrm{Disk}(\mathbb{C}))_{\mathrm{nu}}$-algebra $A$ valued in graded spectra whose value on a disk $D$ is $A(D)_k = \Sigma^{\infty}_+ \bar{A}(D)_k = \Sigma^{\infty}_+ \Gra_{G,\Ran,\leq k\mu }(D\subset \mathbb{C})$.  The following theorem is the essential non-formal input to producing an $\mathbb{E}_2$-algebra.  

\begin{prop}\label{prop:diskequiv}
Suppose that $G$ is reductive.  Let $D\subset \mathbb{C}$ be a disk and $\{x \} \hookrightarrow D$ be a point.  Then, for each $\mu\in \mathbb{X}_{\bullet}^+$, the natural map $\Gra_{G,\Ran, \leq \mu}(\{x\}\subset \mathbb{C}) \to \Gra_{G,\Ran, \leq \mu}(D\subset \mathbb{C})$ induces a homotopy equivalence of spaces.

\end{prop}

The theorem will come down to the following key fact from geometric representation theory, the proof of which appears in \cite{BDQuantization} immediately following \cite[5.3.10]{BDQuantization}:

\begin{lem} \label{lem:grproper}
Suppose that $G$ is reductive.  For each finite set $I$ and each $\mu \in \mathbb{X}_{\bullet}^+$, define $P$ so that the following square is a pullback in the category of topological spaces:
$$
\begin{tikzcd}
P \arrow{r} \arrow{d}{f} & Gr_{G,\text{Ran},\le \mu}(\mathbb{C}) \arrow{d} \\
\mathbb{C}^I \arrow{r} & \text{Ran}(\mathbb{C}).
\end{tikzcd}
$$
Then the map $f$ is a proper map of topological spaces.
\end{lem}

\begin{rmk}
The reader interested in geometric intuition for the above lemma is encouraged to look at \cite[Proposition 1.2.4]{ZhuDemazure}.  The argument in \cite{BDQuantization} demonstrates the stronger statement that the map $f$ is proper before taking complex points.
\end{rmk}

For any $X\subset \mathbb{C}$, $\mathrm{Ran}^{\leq n}(X)$ is a finite colimit of spaces of the form $X^J$ for a finite set $J$.  Thus, we obtain the following corollary:

\begin{cor}\label{cor:ranproper}
Let $X\subset \mathbb{C}$ and $\mu \in \mathbb{X}_{\bullet}^+$.  Then, the restriction $\Gra^{\leq n}_{G,\Ran,\leq \mu}(X\subset \mathbb{C}) \to \mathrm{Ran}^{\leq n}(X)$ of the Beilinson-Drinfeld Grassmannian to $\Ran^{\leq n}(X)$ is proper.
\end{cor}

We now prove Proposition \ref{prop:diskequiv} following ideas of Jacob Lurie.  

\begin{proof}[Proof of Proposition \ref{prop:diskequiv}]
It suffices to show that we have an equivalence upon restricting to $\Ran^{\leq n}$ for each $n$, because $\Gra_{G, \Ran, \leq  \mu}(\{ x\} \subset\mathbb{C}) = \colim_n \Gra^{\leq n}_{G,\Ran, \leq \mu}(\{x \} \subset \mathbb{C})$ and $\Gra_{G, \Ran, \leq  \mu}(D \subset\mathbb{C}) = \colim_n \Gra^{\leq n}_{G,\Ran, \leq \mu}(D \subset \mathbb{C})$ are filtered colimits of closed inclusions.  For ease of notation, for a subset $U\subset \mathbb{C}$, let $\Gra(U)$ stand in for $\Gra^{\leq n}_{G,\Ran, \leq  \mu}(U \subset \mathbb{C})$ for the remainder of the proof.  

We have a pullback square of topological spaces 
\begin{equation*}
\begin{tikzcd}
\Gra (\{ x\}) \arrow[r,"i"]\arrow[d,"p_x"] & \Gra (D) \arrow[d,"p"]\\
\{x \} = \Ran^{\leq n}(\{ x\} )\arrow[r,"j"] & \Ran^{\leq n}(D)
\end{tikzcd} \end{equation*}
where $p$ and $p_x$ are proper by Corollary \ref{cor:ranproper}.  Certainly $\Gra(\{ x\})$ is finite type because it is a filtered piece of the affine Grassmannian.  One consequence of properness is that $\Gra(\mathbb{C})$, and thus $\Gra(D)$, is finite type because $\Ran^{\leq n}(\mathbb{C})$ is a finite colimit of spaces which are proper over spaces of the form $\mathbb{A}^n_{\mathbb{C}}$.  It therefore suffices to show that the upper map $i$ induces an isomorphism on cohomology.  

Applying the proper base change theorem to the constant sheaf $\underline{\mathbb{Z}}$ on $\Gra(D)$, we find that there is a natural equivalence $j^*p_*(\underline{\mathbb{Z}}) \xrightarrow{\sim} (p_x)_{*}i^*(\underline{\mathbb{Z}}).$  The right-hand side is $C^*(\Gra(\{x\}) ;\mathbb{Z})$, the integral cohomology of the fiber at $x$.  On the other hand, the left-hand side is by definition $$\colim_{x\in U} C^*(p^{-1}(U);\mathbb{Z})$$ where the colimit is taken over $U\subset \Ran^{\leq n}(D)$ containing $x$.  We may compute this colimit by extracting a cofinal sequence of opens $$x \subset \cdots \subset \Ran^{\leq n}(D_k) \subset \cdots \subset \Ran^{\leq n}(D_0)$$ where $D_i$ are shrinking disks such that $D_0=D$.  The transition maps in this system induce equivalences because there is a deformation retraction of $D_k$ onto $D_{k+1}$.  Therefore, we conclude that $$\colim_{x\in U} C^*(p^{-1}(U);\mathbb{Z}) \simeq C^*(\Gra(D); \mathbb{Z})$$ and so the map $i$ induces a cohomology isomorphism as desired.  
\end{proof} 

With Proposition \ref{prop:diskequiv} in hand, we may now prove Theorem \ref{thm:schubertE2}.

\begin{proof}[Proof of Theorem \ref{thm:schubertE2}]
Let $A$ be the $\N (\mathrm{Disk}(\mathbb{C}))_{\mathrm{nu}}$-algebra we have been considering, where $A(D)_k = \Sigma^{\infty}_+ \Gra_{G,\Ran,\leq k\mu}(D\subset \mathbb{C})$.  Let $D \subset D'$ be an inclusion of disks and let $\{x \} \hookrightarrow D$ be any point of $D$.  Then, we have for all $k$ the commutative triangle
\begin{equation*}
\begin{tikzcd}
 & \Gra_{G,\Ran, \leq k\mu}(D\subset \mathbb{C})\arrow[d] \\
\Gra_{G,\Ran,\leq k\mu}(\{x\}\subset \mathbb{C}) \arrow[ru]\arrow[r] &\Gra_{G,\Ran,\leq k\mu}(D'\subset \mathbb{C}) \\
\end{tikzcd}
\end{equation*}

By Proposition \ref{prop:diskequiv}, the rightward maps are equivalences.  It follows that the vertical map is an equivalence, and so the natural map $A(D) \to A(D')$ is an equivalence of graded spectra.  Thus, by Theorem \ref{thm:lcdisk}, $A(D)$ acquires the structure of a non-unital $\mathbb{E}_2$ algebra in graded spectra.  Moreover, the underlying graded spectrum is given by $\{ \Sigma^{\infty}_+ \Gra_{G,\leq k \mu}(\mathbb{C})\}_{k \in \mathbb{Z}_{\geq 0}} $.  

The above discussion produces $\{ \Sigma^{\infty}_+ \Gra_{G,\leq k\mu}(\mathbb{C})\}_{k \in \mathbb{Z}_{\geq 0}}$ as a \textit{non-unital} graded $\mathbb{E}_2$-algebra.  To finish the proof of Theorem \ref{thm:schubertE2}, we must provide a unit.  By \cite[Theorem 5.4.4.5]{HA}, it suffices to produce a quasi-unit for $\{ \Sigma^{\infty}_+ \Gra_{G,\leq k\mu}(\mathbb{C})\}_{k \in \mathbb{Z}_{\geq 0}}$, which is a map $\mathbb{S} \longrightarrow  \Sigma^{\infty}_+ \Gra_{G,\leq 0}(\mathbb{C})$ that is both a left and a right unit up to non-canonical homotopy.  We may take, for example, the suspension of the point of the Beilinson-Drinfeld Grassmannian that is the trivial $G$-bundle over $\mathbb{A}^1$ together with its canonical trivialization away from the origin.

\end{proof}

\section{The Bott filtration on \texorpdfstring{$\Omega SU(n)$}{OmegaSU(n)}} \label{sec:MRFil}

In this section we recall and study the Bott filtration \cite{MitchellLoopGroup} on $\Omega SU(n)$.  Our main result is that the Bott filtration is at least $\mathbb{A}_\infty$, meaning in particular that its suspension is an $\mathbb{A}_\infty$-filtered spectrum in the sense of Section \ref{sec:FilGra}.

In the previous Section \ref{sec:Schubert} we learned very general techniques to construct $\mathbb{E}_2$-filtrations.  To see why these techniques do not directly produce the Bott filtration, it is helpful to recall the very instructive Example \ref{sl2example}.  Let us briefly summarize our previous discussion of that example:

\begin{rmk}
Consider $Gr_{SL_2}(\mathbb{C}) \simeq \Omega SU(2) \simeq \Omega S^3$.  This has a natural James filtration
$$* \longrightarrow J_1(S^2) \longrightarrow J_2(S^2) \longrightarrow \cdots \longrightarrow \Omega S^3,$$
which happens to be a special case of the Bott filtration we will define below.
The discussion of Section \ref{sec:Schubert} allows us to prove that the coarsened filtration
$$* \longrightarrow J_2(S^2) \longrightarrow J_4(S^2) \longrightarrow \cdots \longrightarrow \Omega S^3$$
has suspension spectrum a filtered $\mathbb{E}_2$-algebra in spectra.  However, we would like to understand the James filtration, rather than its coarsening, and prove that it is at least an $\mathbb{A}_\infty$-filtration.
\end{rmk}

We will follow Segal \cite{Segal} and access the Bott filtration on $Gr_{SL_n}(\mathbb{C})$ in a somewhat indirect manner, by considering not $Gr_{SL_n}(\mathbb{C})$ but $Gr_{GL_n}(\mathbb{C})$.

We consider $GL_n(\mathbb{C})$ with its usual maximal torus, and choose a Borel such that dominant coweights $\mu$, in bijection with integer sequences $(a_1,a_2,\cdots,a_n)$ such that $a_1 \ge a_2 \ge \cdots \ge a_n$, are given by maps
$$t \stackrel{\mu}{\mapsto} \left( \begin{array}{cccc} t^{a_1} & 0 & \cdots & 0 \\ 0 & t^{a_2} & \cdots & 0 \\ \vdots & \vdots & \ddots & \vdots \\ 0 & 0 & \cdots & t^{a_n} \end{array} \right).$$

\begin{dfn}
Consider the affine Grassmannian $Gr_{GL_n}(\mathbb{C})$.  We denote by $F_{n,k}$ the subset of $Gr_{GL_n}$ that is the closure of the $GL_n(\mathbb{C}[[t]])$ orbit containing:
$$t \mapsto \left( \begin{array}{cccc} t^k & 0 & \cdots & 0 \\ 0 & 1 & \cdots & 0 \\ \vdots & \vdots & \ddots & \vdots \\ 0 & 0 & \cdots & 1 \end{array} \right).$$
Using the language of Section \ref{sec:Schubert}, define for each $k\geq 0$
$$F_{n,k}  :=  \Gra_{GL_n, \leq (k,0,0,\cdots)}(\mathbb{C}).$$
\end{dfn}

The following lemma, affirming a conjecture of Mahowald and Richter, is then an immediate corollary of Theorem \ref{thm:schubertE2} and Remark \ref{rmk:choosespectra}:

\begin{lem} [Conjecture of Mahowald--Richter \cite{MahowaldRichter}] 
The inclusion 
$$\coprod_k F_{n,k} \subset \Omega GL_n(\mathbb{C})$$
may be made into a map of $\mathbb{E}_2$-algebras.  The suspension $$\Sigma^{\infty}_+ \coprod_k F_{n,k}$$ is a graded $\mathbb{E}_2$-algebra in the sense of Section \ref{sec:FilGra}.
\end{lem}

As explained by Segal \cite{Segal}, the coproduct $\coprod_k F_{n,k}$ may be viewed as the subspace of loops in $U(n)$ `of positive winding number.'  The $k$th piece $F_{n,k}$ consists of loops of winding number exactly $k$, and the group completion of $\coprod_k F_{n,k}$ is $\Omega U(n)$.

\begin{exm}
For any $n$, $F_{n,1}$ is equivalent to $\mathbb{CP}^{n-1}$.  The space $F_{2,k}$ is non-canonically homeomorphic to the $k$th stage of the James filtration of $\Omega S^3$, consisting of all words of length $\le k$.
\end{exm}

It is not at all obvious from the above construction that there should exist maps $F_{n,k} \rightarrow F_{n,k+1}$.  To make such maps requires some way of identifying the various connected components of $\Omega U(n)$, each of which is individually equivalent to $\Omega SU(n)$.  Following Segal \cite[pg. 3--4]{Segal}, one makes this identification by multiplying by powers of 
$$\lambda = \left( \begin{array}{cccc} t & 0 & \cdots & 0 \\ 0 & 1 & \cdots & 0 \\ \vdots & \vdots & \ddots & \vdots \\ 0 & 0 & \cdots & 1 \end{array} \right).$$
In other words, there is a map from the space of loops of winding number $k$ to loops of winding number $0$ given by multiplication by $\lambda^{-k}$.

\begin{dfn}
The Bott filtration on $\Omega SL_n(\mathbb{C})$ is the filtration with $k$th piece given by $\lambda^{-k} F_{n,k}$.  We will refer to the associated filtered spectrum 
$$\mathbb{S} \rightarrow \Sigma^{\infty} \lambda^{-1} F_{n,1} \simeq \Sigma^{\infty} \mathbb{CP}^{n-1} \rightarrow \Sigma^{\infty} \lambda^{-2} F_{n,2} \rightarrow \cdots$$
by $\Sigma^{\infty}_+ \{F_{n,k}\}$.
\end{dfn}

The above constructions make $\Sigma^{\infty}_+ \{F_{n,k}\}$ into a filtered spectrum whose underlying graded spectrum is $\mathbb{E}_2$.  We will now discuss the problem of making the filtered spectrum itself $\mathbb{E}_2$, or at least $\mathbb{A}_\infty$.  For this, recall from Lemma \ref{lem:FilAsGrMod} that there is a graded $\mathbb{E}_\infty$ ring $A=\Sigma^{\infty}_+ \mathbb{Z}^{ds}_{\ge 0}$ so that filtered spectra may be described as $A$-modules in graded spectra.  We now discuss the following theorem, which implies Theorem \ref{thm:BottIsAoo} from the Introduction, and which again follows easily from the machinery of Beilinson--Drinfeld Grassmannians:

\begin{thm} \label{thm:AooFil}
There is a map of $\mathbb{E}_2$-algebra objects in graded spectra
$$A \simeq \Sigma^{\infty}_+ \mathbb{Z}^{ds}_{\ge 0} \longrightarrow \Sigma^{\infty}_+ \coprod_k F_{n,k}.$$
In particular, $\Sigma^{\infty}_+ \coprod_k F_{n,k}$ is an $\mathbb{A}_\infty$-algebra in $A$-modules, and so $\Sigma^{\infty}_+ \{F_{n,k}\}$ is a filtered $\mathbb{A}_\infty$-algebra.
\end{thm}

\begin{proof}[Proof of Theorem \ref{thm:AooFil}]
Consider $Gr_{\mathbb{G}_m}(\mathbb{C})$, the affine Grassmannian for the multiplicative group.  This is a model for $\Omega S^1$ and so has $\mathbb{Z}$ many contractible connected components.  Choosing a dominant coweight corresponding to a loop of winding number $1$ identifies a copy of $\mathbb{Z}^{ds}_{\ge 0}$ inside of $Gr_{\mathbb{G}_m}(\mathbb{C})$.  The Beilinson--Drinfeld Grassmannian for the group $G=\mathbb{G}_m$ then describes $\Sigma^{\infty}_+ \mathbb{Z}^{ds}_{\ge 0}$ as a sub-$\mathbb{E}_2$-algebra of $\Sigma^{\infty}_+ Gr_{\mathbb{G}_m}(\mathbb{C})$.

Now, the map of groups $\mathbb{G}_m \rightarrow GL_n(\mathbb{C})$ given by the dominant coweight $(1,0,\cdots,0)$ induces a map of Beilinson--Drinfeld Grassmannians.  Applying Remark \ref{rmk:bdgrfunct}, we obtain the desired map of graded $\mathbb{E}_2$-algebras.
\end{proof}

\begin{rmk}
The $\mathbb{E}_2$-algebra map $A \rightarrow \Sigma^{\infty}_+ \coprod_k F_{n,k}$ sits in a commutative diagram of $\mathbb{E}_2$-algebras
$$
\begin{tikzcd}
A \arrow{d} \arrow{r} & \Sigma^{\infty}_+ \coprod_k F_{n,k} \arrow{d} \\
\Sigma^{\infty}_+ \mathbb{Z} \arrow{r} & \Sigma^{\infty}_+ \Omega U(n).
\end{tikzcd}
$$
The map $\Sigma^{\infty}_+ \mathbb{Z} \rightarrow \Sigma^{\infty}_+ \Omega U(n)$ may be described as the suspension of the natural map
$$\Omega^2(BU(1) \rightarrow BU(n)).$$
\end{rmk}

\begin{rmk} \label{rmk:E2fil}
The fact that there is an $\mathbb{E}_2$-algebra map $A \rightarrow \Sigma^{\infty}_+ \coprod_k F_{n,k}$ is stronger than the fact that $\Sigma^{\infty}_+ \{F_{n,k}\}$ is $\mathbb{A}_\infty$ filtered, but it is weaker than the claim that $\Sigma^{\infty}_+ \{F_{n,k}\}$ is $\mathbb{E}_2$ filtered.  We do not know if the Bott filtration is $\mathbb{E}_2$ or not, but would be very interested to learn the answer.

The machinery of Beilinson--Drinfeld Grassmannians proves that the coarsened filtration consisting of every $n$th piece of the Bott filtration (i.e. $\Sigma^{\infty}_+ \{F_{n,nk}\}$) is an $\mathbb{E}_2$-filtration.  The question is equivalent to the production of an $\mathbb{E}_3$-algebra map from $A$ to the $\mathbb{E}_3$-center of the $\mathbb{E}_2$-algebra $\Sigma^{\infty}_+ \coprod \{F_{n,k}\}$. 
After group completion, this would in particular imply the existence of an $\mathbb{E}_3$-algebra map
$$\mathbb{Z} \longrightarrow (\Omega U(n))^{hU(n)}.$$
We do not know whether even this last map exists.
\end{rmk}

We end this section by sketching what is named Construction \ref{cnstr:IntroGr} in the Introduction:

\begin{cnstr} \label{cnstr:E2GrConstruction}
The graded $\mathbb{A}_\infty$-algebra $gr(\Sigma^{\infty}_+ \{F_{n,k}\})$ may be equipped with the structure of a graded $\mathbb{E}_2$-algebra.
\end{cnstr}

\begin{proof}[Proof sketch]
As explained above, the $\mathbb{E}_2$-algebra in spaces $\coprod F_{n,k}$ receives a natural $\mathbb{E}_2$-map from $\mathbb{Z}^{ds}_{\ge 0}$.  We may thus view $\coprod F_{n,k}$ as an $\mathbb{E}_2$-algebra over $\mathbb{Z}^{ds}_{\ge 0}$ (by, e.g., the straightening and unstraightening correspondence).  There is a diagram of $\mathbb{E}_2$-algebras:
$$
\begin{tikzcd}
\coprod_k F_{n,k} \arrow{d} \arrow{r} & \Omega U(n) \arrow{r} & \Omega U \simeq BU \times \mathbb{Z} \arrow{r}{J} & Pic(\mathbb{S}) \subset \Sp \\
\mathbb{Z}^{ds}_{\ge 0}.
\end{tikzcd}
$$
In \cite[1.7]{Segal}, it is proven that the colimit of the functor $\coprod F_{n,k} \longrightarrow \Sp$ is equivalent (as a spectrum) to $gr(\Sigma^{\infty}_+ \{F_{n,k}\})$.  Note that this colimit is more classically described as the Thom spectrum of the map $\coprod F_{n,k} \longrightarrow BU \times \mathbb{Z}$.

One may also compute this colimit by first making a left Kan extension along the map $$\coprod F_{n,k} \rightarrow \mathbb{Z}^{ds}_{\ge 0},$$ and then taking the coproduct of the images of the resulting map $$\mathbb{Z}^{ds}_{\ge 0} \rightarrow \Sp.$$  Taking an operadic left Kan extension as in \cite[3.1.2]{HA}, one learns that the left Kan extension $\mathbb{Z}^{ds}_{\ge 0} \longrightarrow \Sp$ is lax $\mathbb{E}_2$-monoidal.  The properties of Day convolution (explained in, e.g., Appendix \ref{app:day}) then imply that the Thom spectrum is naturally an $\mathbb{E}_2$-graded spectrum.

To see that the underlying $\mathbb{A}_\infty$-graded spectrum agrees with the associated graded of the Bott filtration, notice that the zero-section of the Thom construction is a map of graded $\mathbb{E}_2$-algebras.  This zero-section is, on the $k$th graded piece, a model for the map $$\Sigma^{\infty} F_{n,k} \longrightarrow \Sigma^{\infty} F_{n,k}/F_{n,k-1}.$$
The sequence of graded $\mathbb{E}_2$-algebra maps
$$\Sigma^{\infty}_+ \mathbb{Z}^{ds}_{\ge 0} \longrightarrow \Sigma^{\infty}_+ \coprod_k F_{n,k} \longrightarrow \text{Thom}\left( \coprod_k F_{n,k} \right)$$
then implies the result.
\end{proof} 

\section{Multiplicative Aspects of Weiss Calculus} \label{sec:MultWeiss}

In this section, we determine the multiplicative properties of the Weiss calculus polynomial approximation functors \cite{Weiss}.  More precisely, for a functor $F$, we aim to understand the Taylor tower of $F\wedge F$ in terms of the tower for $F.$  The results in this section are likely known to experts, but the authors were not able to locate it in the literature.  They thank Jacob Lurie for suggesting that Theorem \ref{thm:weissmonoidal} is true.

\subsection{Review of Weiss calculus}
We briefly review notions of Weiss calculus to set notation.  The reader is referred to \cite{Weiss} for proofs and additional details.  We note that the discussion there is in the case of real vector spaces, but the results work just the same in the complex case.  We shall also work in the language of $\infty$-categories rather than topological categories, and Remark \ref{rmk:infinityweiss} justifies this passage.  

Let $\J$ be the $\infty$-category which is the nerve of the topological category whose objects are finite dimensional complex vector spaces equipped with a Hermitian inner product and whose morphisms are spaces of linear isometries.  

Weiss calculus studies functors out of $\J$ in a way analogous to Goodwillie calculus, by understanding successive ``polynomial approximations'' to these functors.  Here, we will discuss only the stable setting where we apply the theory to the functor category $\Sp^{\J}$. The central definition is:

\begin{dfn}\label{dfn:polyfun}
A functor $F\in \Sp^{\J}$ is polynomial of degree $n$ if the natural map $$F(V) \to \lim_U F(U\oplus V)$$ is an equivalence, where the limit is indexed over the $\infty$-category of nonzero subspaces $U\subset \mathbb{C}^{n+1}.$
\end{dfn}

As in Goodwillie calculus, the inclusion of the full subcategory $\Poly^{\leq n}(\Sp^{\J}) \subset \Sp^{\J}$ of functors which are polynomial of degree $n$ admits a left adjoint $$P_n: \Sp^{\J} \xrightleftharpoons{\quad} \Poly^{\leq n}(\Sp^{\J}): j_n.$$ 
 The unit $\eta_n$ of this adjunction provides for each $F\in \Sp^{\J}$ a natural transformation $F \to P_nF$ which we will refer to as the \emph{degree $n$ polynomial approximation} of $F$. 

\begin{rmk}\label{rmk:infinityweiss}
This universal property was not explicitly stated in \cite{Weiss}, but it follows formally from Weiss's results as follows: the functor $P_n$ and the transformation $\eta_n$ can be defined explicitly as in \cite{Weiss} by iteratively applying the functor $\tau_n: \Sp^{\J} \to \Sp^{\J}$ defined by the formula $$\tau_n F(V) = \lim_U F(U\oplus V)$$ with the limit indexed as in Definition \ref{dfn:polyfun}.   The facts required of the functors $P_n$ in the proof of Theorem 6.1.1.10 in \cite{HA} are precisely the content of Theorem 6.3 of \cite{Weiss}.  
\end{rmk}

Given this universal property, Proposition 5.4 of \cite{Weiss} ensures the existence of a natural Taylor tower $$F \longrightarrow \cdots \longrightarrow P_{n} F \xrightarrow{p_{n-1}} P_{n-1} F \longrightarrow \cdots \longrightarrow P_0F$$ living under any functor $F\in \Sp^{\J}.$  The fiber $D_n F$ of $p_{n-1}$ has the special property that it is polynomial of degree $n$ and $P_{n-1} D_n F \simeq 0$.  Such a functor is called \emph{$n$-homogeneous}; such functors are completely classified by the following theorem:

\begin{thm}[{{\cite[Theorem 7.3]{Weiss}}}]
Let $F\in \Sp^{\J}$.  Then $F$ is an $n$-homogeneous functor if and only if there exists a spectrum $\Theta$ with an action of the unitary group $U(n)$ such that $$F(V) = (\Theta \wedge S^{nV})_{hU(n)}.$$
\end{thm}

\subsection{The Taylor tower}

It is helpful, for our study of multiplicative properties, to package all polynomial approximations into a single object--the following construction makes this precise:



\begin{cnstr}\label{cnstr:tower}
We now construct a functor $$\mathrm{Tow}: \Sp^{\J} \to \Cofil(\Sp^{\J})$$ with the property that it sends a functor $F\in \Sp^{\J}$ to its Taylor tower $$\mathrm{Tow}(F) = P_0F \longleftarrow P_1F \longleftarrow P_2F \longleftarrow \cdots.$$

Recall that the $P_n$ functors are given as left adjoints of the fully faithful inclusions $$\Poly^{\leq n}(\Sp^{\J}) \subset \Sp^{\J}.$$  We proceed by telling a parametrized version of this story that includes all $n$ simultaneously.  The proper framework for such a story is the formalism of \emph{relative adjunctions}; these are developed in the $\infty$-categorical context in \cite[Section 7.3.2]{HA}.  

Consider the category $\Sp^{\J}\times \Z^{op}_{\geq 0}$ together with the full subcategory $$(\Sp^{\J}\times \Z^{op}_{\geq 0})_{\text{poly}} \subset \Sp^{\J}\times \Z^{op}_{\geq 0}$$ on the pairs $(F, [n])$ such that $F\in \Poly^{\leq n}(\Sp^{\J}).$  Via projection, these fit into a diagram
$$
\begin{tikzcd}\label{dia:reladj}
\Sp^{\J}\times \Z^{op}_{\geq 0} \arrow[rd,"q"]& &(\Sp^{\J}\times \Z^{op}_{\geq 0})_{\text{poly}} \arrow[ld,"p"]\arrow[ll,"i"]  \\
& \Z^{op}_{\geq 0}
\end{tikzcd}
$$
 This will be relevant to us because the category of sections of $q$ are precisely $\Cofil(\Sp^{\J}).$  The sections of $p$ can be thought of as those cofiltered functors such that the $n$th piece is polynomial of degree $n$.  We will denote this category of sections of $p$ by $\Cofil(\Sp^{\J})_{\text{poly}}.$ 

On the fibers over an integer $[n] \in \Z^{op}_{\geq 0}$, we see the inclusion $\Sp^{\J} \leftarrow \Poly^{\leq n}(\Sp^{\J}).$  It is in this sense that the current picture is a parametrized version of the ordinary polynomial approximations.  We now claim that $i$ admits a left adjoint $P^{\text{total}}: \Sp^{\J}\times \Z^{op}_{\geq 0} \to (\Sp^{\J}\times \Z^{op}_{\geq 0})_{\text{poly}}$ \emph{relative} to $\Z^{op}_{\geq 0}.$    The strategy is to use Proposition 7.3.2.6 of \cite{HA}, which tells us that we need to check the following three statements:
\begin{enumerate}
\item The functors $p$ and $q$ are locally Cartesian categorical fibrations.
\item For each $[n]\in \Z^{op}_{\geq 0}$, the functor on fibers $i|_{p^{-1}[n]}:p^{-1}[n] \to q^{-1}[n]$ admits a right adjoint.  
\item The functor $i$ carries locally $p$-Cartesian morphisms of $(\Sp^{\J}\times \Z^{op}_{\geq 0})_{\text{poly}}$ to locally $q$-Cartesian morphisms of $\Sp^{\J}\times \Z^{op}_{\geq 0}$.
\end{enumerate}

Condition (2) is clear from the existence of polynomial approximations in Weiss calculus.  To see conditions (1) and (3), we first note that $q$ is in fact a Cartesian fibration because it is a projection from a product.  Moreover, the $q$-Cartesian morphisms are precisely those morphisms which are equivalences on the $\Sp^{\J}$ coordinate.  
Now suppose we are given a pair $(F, [m]) \in \Sp^{\J}\times \Z^{op}_{\geq 0}$ such that $F\in \Poly^{\leq m}(\Sp^{\J})$ and morphism $\sigma :[n]\to [m]$.  Any $q$-Cartesian edge lying over $\sigma$ with target $(F, [m])$ has source equivalent to $(F, [n])$ and thus is also in the full subcategory $(\Sp^{\J}\times \Z^{op}_{\geq 0})_{\text{poly}}$ because $m\leq n$.   
Since $p$ is certainly an inner fibration (by construction as a full subcategory), this implies that $p$ is also a Cartesian fibration and that the inclusion $i$ carries $p$-Cartesian edges to $q$-Cartesian edges.  Since any Cartesian fibration is a categorical fibration (\cite[Proposition 3.3.1.7]{HTT}), conditions (1) and (3) are verified.  

We now wish to look at the adjunction at the level of sections of $q$ and $p$.  Considering functors from $\Z_{\geq 0}^{op}$ into Diagram \ref{dia:reladj}, we obtain a new diagram 
$$
\begin{tikzcd}
\Fun(\Z^{op}_{\geq 0},\Sp^{\J}\times \Z^{op}_{\geq 0}) \arrow[rr, bend left=10,"P^{\text{total}}_*"] \arrow[rd,"q_*"]& &\Fun(\Z^{op}_{\geq 0},(\Sp^{\J}\times \Z^{op}_{\geq 0})_{\text{poly}}) \arrow[ld,"p_*"]\arrow[ll,"i_*"]  \\
& \Fun(\Z^{op}_{\geq 0}, \Z^{op}_{\geq 0})
\end{tikzcd}
$$
which exhibits $P^{\text{total}}_*$ as a left adjoint of $i_*$ relative to $\Fun(\Z_{\geq 0}^{op},\Z_{\geq 0}^{op}).$  Proposition 7.3.2.5 of \cite{HA} ensures that there is an adjunction at the level of fibers above $\text{id}\in \Fun(\Z_{\geq 0}^{op},\Z_{\geq 0}^{op})$: $$\mathcal{P}: \Cofil(\Sp^{\J}) \xrightleftharpoons{\quad} \Cofil(\Sp^{\J})_{\text{poly}}:j .$$  
Finally, observe that the unique functor $r: \Z^{op}_{\geq 0} \to *$ induces an adjunction $$r^*: \Sp^{\J} \xrightleftharpoons{\quad} \Cofil(\Sp^{\J}): \lim$$ where $r^*$ is the constant functor and $\lim$ is the same as right Kan extension along $r$.  We now compose these adjunctions, denoting $\mathrm{Tow} = \mathcal{P}\circ r^*$ to obtain: $$\mathrm{Tow}: \Sp^{\J} \xrightleftharpoons{\quad}  \Cofil(\Sp^{\J})_{\text{poly}}: \lim.$$

By construction, $\mathrm{Tow}(F)$ is the cofiltered spectrum $$P_0 F\longleftarrow P_1 F \longleftarrow P_2 F \longleftarrow \cdots .$$  This concludes the construction of $\mathrm{Tow}.$
\end{cnstr}

\subsection{Multiplicativity of Tow}
The next task is to understand the multiplicative structure of $\mathrm{Tow}$.  The idea is that we would like to express $\mathrm{Tow}(F\wedge F)$ in terms of $\mathrm{Tow}(F)$ and the Day convolution monoidal structure on $\Cofil(\Sp^{\J}).$  We start with the following lemma:  

\begin{lem}
The Weiss tower functor $\mathrm{Tow}$ defines an oplax symmetric monoidal functor $$\mathrm{Tow}: \Sp^{\J} \to \Cofil(\Sp^{\J}).$$
\end{lem}
\begin{proof}
Recall that the Weiss tower functor was defined as a composite $\mathrm{Tow} = \mathcal{P} \circ r^*.$  The functor $r^*$ is just the constant functor, so it is symmetric monoidal.  On the other hand, $\mathcal{P}$ is adjoint to the inclusion $j: \Cofil(\Sp^{\J})_{\text{poly}} \to \Cofil(\Sp^{\J}).$  Since the class of functors which are polynomial of degree $n$ is closed under finite limits, the subcategory $(\Cofil(\Sp^{\J}))_{\text{poly}}$ is closed under the convolution tensor product.  We may therefore give it the structure of a symmetric monoidal $\infty$-category such that $j$ is symmetric monoidal.  This makes the left adjoint $\mathcal{P}$ an oplax symmetric monoidal functor, which induces an oplax symmetric monoidal structure on $\mathrm{Tow}.$  \end{proof}

Concretely, this oplax structure can be described on the $n$th filtered piece as follows: suppose $F,G \in \Sp^{\J}$; since $\mathrm{Tow}(F) \otimes \mathrm{Tow}(G) \in  \Cofil(\Sp^{\J})_{\text{poly}}$, the filtered piece $(\mathrm{Tow}(F) \otimes \mathrm{Tow}(G))_n$ is polynomial of degree $n$.  It follows that the natural map from $F\wedge G$ factors through a map $\varphi_n: P_n(F\wedge G)\to (\mathrm{Tow}(F) \otimes \mathrm{Tow}(G))_n$.  

In order to show that $\mathrm{Tow}$ is a symmetric monoidal functor, one would need to show that each $\varphi_n$ is an equivalence for all $n$.  We will do this after restricting to a smaller subcategory of functors with nice convergence properties:

\begin{dfn}
Let $F\in \Sp^{\J}$ be a functor.  Call $F$ \emph{rapidly convergent} if $F$ takes values in connective spectra and there exist real numbers $c,\alpha>0$ such that the natural map $F(W) \to P_nF(W)$ is $(\alpha n)\text{dim}(W) - c$ connected.  We denote by $\Sp^{\J}_{\text{conv}}$ the category of rapidly convergent functors.  
\end{dfn}

\begin{exm} \label{ex:aronefunctor}
Let $V\in \J$ be a complex vector space.  
The functor $F_V\in \Sp^{\J}$ defined by $$F_V(W) = \Sigma^{\infty}_+ \Omega \J(V, V\oplus W)$$ is rapidly convergent.  Indeed, \cite{Arone} shows that its homogeneous layers are given by $$D_nF_V(W) = \Omega^{\infty}(s_n(V) \smsh S^{nW})_{hU(n)}$$ where $s_n(V)$ is the suspension spectrum of a space.  Since colimits do not lower connectivity, this implies $D_nF_V(W)$ is at least $(2n)\text{dim}(W)-1$-connected.  It follows from the Milnor sequence that $F_V$ is rapidly convergent, where $\alpha$ can be taken to be $2$.  
\end{exm}

Observe that rapidly convergent functors in particular have convergent Weiss towers.  However, the following lemma of Weiss about functors ``agreeing up to order $n$'' allows us to say more:

\begin{lem}[\cite{WeissErratum}]\label{lem:ordernagree}
Let $F,G\in \Sp^{\J}$ be functors, $\eta: F\to G$ a natural transformation, and $n\geq 0$ an integer.  Suppose that there exists $c>0$ such that for all $W\in \J$, the map of spectra $\eta_W: F(W) \to G(W)$ is $(n+1)\text{dim}(W) -c$ connected.  Then the natural transformation $P_n\eta: P_n F\to P_n G$ is an equivalence.  
\end{lem}

The following corollary is immediate:

\begin{cor} \label{cor:rapidconv}
Let $F,G\in \Sp^{\J}_{\text{conv}}$ be rapidly convergent and $n>0$ an integer.  Then there exists an integer $M$ such that for $m>M$, the natural transformation $F\smsh G \to P_m F\smsh P_m G$ is an equivalence after applying $P_k$ for all integers $0\leq k \leq n.$  
\end{cor}

We will show that this implies the following further corollary:

\begin{cor}\label{cor:varphiequiv}
The map $\varphi_n$ constructed above is an equivalence for all $n$ when $F,G\in \Sp^{\J}_{\text{conv}}$ are rapidly convergent functors.  
\end{cor}
The proof will require the following basic lemma whose proof we will record at the end of Appendix \ref{sect:monoidal2}.  
\begin{lem}\label{lem:cubes}
Let $X$, $Y \in \Cofil(\Sp^{\J})$ and $n>0$ an integer.  Then we have the following formula for the successive fibers: $$\fib ((X\otimes Y)_{n} \to (X\otimes Y)_{n-1}) \simeq \coprod_{p+q=n} \fib (X_p\to X_{p-1}) \wedge \fib (Y_q \to Y_{q-1}).$$
\end{lem}

We now prove the corollary:

\begin{proof}[Proof of Corollary \ref{cor:varphiequiv}]
Corollary \ref{cor:varphiequiv} implies that by replacing $F$ and $G$ by appropriate polynomial approximations, it suffices to consider the case where $F$ and $G$ are polynomial of degree $m$ for some $m$ (and thus, have finite Weiss towers).   Note further that Lemma \ref{lem:cubes} applied to the case $X = \mathrm{Tow}(F)$, $Y=\mathrm{Tow}(G)$ implies that $$\text{fib}((\mathrm{Tow}(F)\otimes \mathrm{Tow}(G) )_{n+1} \to (\mathrm{Tow}(F) \otimes \mathrm{Tow}(G))_n)$$ is homogeneous of degree $n+1$.  It follows that the fiber of the natural map $$F\smsh G \to (\mathrm{Tow}(F)\otimes \mathrm{Tow}(G))_n$$  is a finite limit of functors killed by $P_n$.  Since $P_n$ commutes with finite limits, we conclude that the natural map $\varphi_n: P_n(F\smsh G)\to (\mathrm{Tow}(F)\otimes \mathrm{Tow}(G))_n$ is an equivalence.  
\end{proof}

The proof shows further that rapidly convergent functors are closed under the tensor product.  In total, we have now proven the following theorem:

\begin{thm}\label{thm:weissmonoidal}
The Weiss tower defines a symmetric monoidal functor $$\mathrm{Tow}: \Sp^{\J}_{\text{conv}} \to \Cofil(\Sp^{\J}_{\text{conv}}).$$
\end{thm}
\begin{rmk}\label{rmk:goodwilliecase}
Theorem \ref{thm:weissmonoidal} and its proof work equally well in Goodwillie calculus.  There, the hypothesis on convergence can be removed, and the adjoining discussion can be replaced with the observation \cite[Lemma 6.10]{GoodwillieIII} that the smash product of an $n$-reduced functor and an $m$-reduced functor is $(n+m)$-reduced.  In Weiss calculus, this fact is also true but not in the literature, so we have opted to give the above proof which is sufficient for our application.  
\end{rmk}

Combining this with Example \ref{ex:aronefunctor}, we obtain:

\begin{cor}\label{cor:aronemonoidal}
Let $V\in \J$ be a complex vector space.  
The functor $F_V\in \Sp^{\J}$ defined by $$F_V(W) = \Sigma^{\infty}_+ \Omega \J(V, V\oplus W)$$ determines a cofiltered associative algebra $\mathrm{Tow}(F_V) \in \Alg_{\mathbb{A}_{\infty}}(\Cofil(\Sp^{\J})).$
\end{cor}

\section{Stable \texorpdfstring{$\mathbb{A}_\infty$}{Aoo} Splittings} \label{sec:AooSplit}

In this brief section, we assemble results proved above in order to obtain what is labeled Theorem \ref{thm:MainAoo} in the Introduction:

\begin{thm} \label{thm:MainAooInText}
As an $\mathbb{A}_\infty$-algebra object in filtered spectra, the Bott filtration of $\Sigma^{\infty}_+ \Omega SU(n)$ is equivalent to its associated graded.
\end{thm}

\begin{proof}
The construction of the stable Bott filtration as an $\mathbb{A}_\infty$-filtered spectrum is our Theorem \ref{thm:AooFil}.  According to our Theorem \ref{thm:SplitMachine}, to complete the proof of Theorem \ref{thm:MainAooInText} it suffices to produce an $\mathbb{A}_\infty$-cofiltered spectrum with limit $\Sigma^{\infty}_+ \Omega SU(n)$ and with the property that certain composites are equivalences.

As explained in the Introduction, we follow Arone \cite{Arone} in producing the desired cofiltered spectrum by means of Weiss calculus.  Consider, in the notation of Section \ref{sec:MultWeiss} and particularly Example \ref{ex:aronefunctor}, the functor $$F_V:\mathcal{J} \longrightarrow \Sp$$ given by $F_V(W) = \Sigma^{\infty}_+ \Omega \mathcal{J}(V, V \oplus W)$.  Corollary \ref{cor:aronemonoidal} implies that the Taylor tower of $F_V$, applied to $W$, is an $\mathbb{A}_\infty$-cofiltered spectrum with limit $\Sigma^{\infty}_+ \Omega \mathcal{J}(V,V \oplus W)$.

Specializing to the case $V=\mathbb{C}^{n-1}$, $W=\mathbb{C}$, it is straightforward to see that $\mathcal{J}(\mathbb{C}^{n-1},\mathbb{C}^{n-1} \oplus \mathbb{C})$ is equivalent to $SU(n)$.  Roughly speaking, this is because any embedding of $\mathbb{C}^{n-1}$ into $\mathbb{C}^n$ may be extended in a unique way to an automorphism of $\mathbb{C}^n$ that is unitary of determinant one.  Thus, applying Corollary \ref{cor:aronemonoidal} in the case $V=\mathbb{C}^{n-1}$, $W=\mathbb{C}$ gives an $\mathbb{A}_\infty$-cofiltered spectrum with the desired limit.  To complete the proof of Theorem \ref{thm:MainAooInText} it suffices then to check that certain composites are equivalences.  In fact, one of the main results of \cite{Arone} is that those composites are equivalences (see the proof of \cite[Theorem 1.2]{Arone}).
\end{proof}

\begin{rmk}
Mitchell and Richter constructed \cite{CrabbBarcelona} a filtration not just of $$\Omega SU(n) \simeq \Omega \mathcal{J}(\mathbb{C}^{n-1},\mathbb{C}^n),$$ but also of $\Omega \mathcal{J}(V,V\oplus W)$ for a general $V$ and $W$.  Arone showed in \cite[Theorem 1.2]{Arone} that this Mitchell--Richter filtration always stably splits, and Corollary \ref{cor:aronemonoidal} provides an $\mathbb{A}_\infty$-cofiltered spectrum inducing this splitting.  We do not know, however, whether the Mitchell--Richter filtration is always $\mathbb{A}_\infty$-split because it is not known whether the filtration is $\mathbb{A}_\infty$.   
\end{rmk}

\section{\texorpdfstring{$\mathbb{E}_2$}{E2} Splittings in Complex Cobordism} \label{sec:MUE2}

We remark in this section that the $\mathbb{A}_\infty$ splitting $$\Sigma^{\infty}_+ \Omega SU(n) \simeq gr(\Sigma^{\infty}_+ \{F_{n,k}\})$$ becomes $\mathbb{E}_2$ after base-change to complex bordism.  More precisely, suppose that $gr(\Sigma^{\infty}_+ \{F_{n,k}\})$ is equipped with some graded $\mathbb{E}_2$-ring structure extending the natural graded $\mathbb{A}_\infty$-ring structure.  Construction \ref{cnstr:E2GrConstruction} provides one possible way to do this.  We will by abuse of notation use $gr(\Sigma^{\infty}_+ \{F_{n,k}\})$ also to denote the underlying \textit{ungraded} $\mathbb{E}_2$-algebra, and our main theorem is that there is an equivalence of (ungraded) $\mathbb{E}_2$-$MU$-algebras:

$$MU \smsh \Sigma^{\infty}_+ \Omega SU(n) \simeq MU \smsh gr(\Sigma^{\infty}_+ \{F_{n,k}\}).$$

Notice that the results of Section \ref{sec:AooSplit} give an $\mathbb{A}_\infty$-equivalence between these two $MU$-algebras.  This $\mathbb{A}_\infty$-$MU$-algebra equivalence is adjoint to a map of $\mathbb{A}_\infty$-$\mathbb{S}$-algebras
\begin{equation} \label{SplittingMap}
\Sigma_+^{\infty} \Omega SU(n) \longrightarrow MU \smsh gr(\Sigma^{\infty}_+ \{F_{n,k}\}).
\end{equation}

Our task in this section will be to show that (\ref{SplittingMap}) may be refined to a morphism of $\mathbb{E}_2$-ring spectra.  We do so by obstruction theory--the key fact powering our proof is that 
$$MU_{2*+1}\left(\Omega SU(n)\right) \cong 0.$$
This classical vanishing result may be proven via Atiyah--Hirzerburch spectral sequence, using the even cell-decomposition of $Gr_{SL_n}(\mathbb{C})$.

Inspired by \cite{ChadwickMandell}, we prove the following general result (implying in particular Theorem \ref{thm:MainMUE2}):

\begin{thm}
Suppose that $R$ is an $\mathbb{E}_2$-ring spectrum with no homotopy groups in odd degrees.  Then any homotopy commutative ring homomorphism
$$\Sigma^{\infty}_+ \Omega SU(n) \rightarrow R$$
lifts to a morphism of $\mathbb{E}_2$-ring spectra.  Moreover, any chosen $\mathbb{A}_\infty$ lift may be extended to an $\mathbb{E}_2$ lift.
\end{thm}

\begin{proof} 
By taking connective covers, one learns that any ring homomorphism
$$\Sigma^{\infty}_+ \Omega SU(n) \rightarrow R$$
must factor through the natural $\mathbb{E}_2$-algebra map $\tau_{\ge 0} R \rightarrow R$.  Thus, without loss of generality we will assume that $R$ is $(-1)$-connected.

It is clear that the composite ring homomorphism
$$\Sigma^{\infty}_+ \Omega SU(n) \longrightarrow R \longrightarrow \tau_{\le 0} R \simeq H\pi_0(R)$$
may be lifted to an $\mathbb{E}_2$-ring homomorphism factoring through $\tau_{\le 0} \Sigma^{\infty}_+ \Omega SU(n) \simeq H\mathbb{Z}$.   Suppose now for $q>0$ that we have chosen an $\mathbb{E}_2$-ring homomorphism 
$$\Sigma^{\infty}_+ \Omega SU(n) \longrightarrow \tau_{\le q-1} R$$
We will show that there is no obstruction to the existence of a further $\mathbb{E}_2$-lift $$\Sigma^{\infty}_+ \Omega SU(n) \longrightarrow \tau_{\le q} R,$$
and that one may be chosen lifting any specified $\mathbb{A}_\infty$ map $\Sigma^{\infty}_+ \Omega SU(n) \rightarrow \tau_{\le q} R$.

According to \cite[Theorem $4.1$]{ChadwickMandell}, there is a diagram of principal fibrations
$$
\begin{tikzcd}
\mathbb{E}_2\text{-Ring}(\Sigma^{\infty}_+ \Omega SU(n), \tau_{\le q} R) \arrow{r} \arrow{d} & \mathbb{A}_\infty\text{-Ring}(\Sigma^{\infty}_+ \Omega SU(n), \tau_{\le q} R) \arrow{d} \\
\mathbb{E}_2\text{-Ring}(\Sigma^{\infty}_+ \Omega SU(n), \tau_{\le q-1} R) \arrow{r} \arrow{d} & \mathbb{A}_\infty\text{-Ring}(\Sigma^{\infty}_+ \Omega SU(n), \tau_{\le q-1} R) \arrow{d} \\
\cS_*(BSU(n),K(\pi_q R,q+3)) \arrow{r} & \cS_*(SU(n),K(\pi_q R,q+2))
\end{tikzcd}
$$
For $q$ odd, $\tau_{\le q-1} R \simeq \tau_{\le q} R$, so there is no obstruction.  Let us therefore assume that $q$ is even.

Since the cohomology of $BSU(n)$ is even-concentrated with coefficients in any abelian group, we have that $\pi_0 \cS_*(BSU(n),K(\pi_q R,q+3)) \cong H^{q+3}(BSU(n);\pi_q R)$ is zero.  It follows then that the given class $$x \in \pi_0 \mathbb{E}_2\text{-Ring}(\Sigma^{\infty}_+ \Omega SU(n), \tau_{\le q-1} R)$$ admits some lift $$\widetilde{x} \in \mathbb{E}_2\text{-Ring}(\Sigma^{\infty}_+ \Omega SU(n), \tau_{\le q} R).$$  We may need to modify $\widetilde{x}$ to match our chosen $\mathbb{A}_\infty$-ring homomorphism.  This is always possible so long as the map
$$\pi_1(\cS_*(BSU(n),K(\pi_q R,q+3))) \longrightarrow \pi_1(\cS_*(SU(n),K(\pi_q R,q+2)))$$
is surjective.  Said in other terms, this is just the map
$$H^{2q+2}(BSU(n);\pi_q R) \longrightarrow H^{2q+1}(SU(n);\pi_q R) \cong H^{2q+2}(\Sigma SU(n);\pi_q R)$$
induced by the natural map $\Sigma SU(n) \rightarrow BSU(n)$.  It is a classical fact that this map is surjective (it follows from a calculation with the bar spectral sequence, using the fact that the cohomology of $SU(n)$ is exterior).
\end{proof}

\section{Obstructions to a General \texorpdfstring{$\mathbb{E}_2$}{E2} Splitting} \label{sec:Obstruction}

Let $3< n\leq \infty$ be an integer.  The $\mathbb{A}_{\infty}$ filtered equivalence of Theorem \ref{thm:MainAoo} gives an equivalence of $\mathbb{A}_\infty$ ring spectra  $$\Sigma^{\infty}_+ \Omega SU(n) \simeq gr(\Sigma^{\infty}_+ \{ F_{n,k} \}).$$  The right-hand side is the associated graded of the stable Bott filtration $\Sigma^{\infty}_+ \{ F_{n,k} \}$, which we showed is $\mathbb{A}_\infty$ in Theorem \ref{thm:BottIsAoo}, but which is not known to be $\E_2$ (see Question \ref{qst:BottE2}).

In this section, we show that the graded spectrum on the right-hand side cannot be given a graded $\E_2$ structure which makes the above equivalence $\E_2$ on underlying ring spectra.  This proves Theorem \ref{thm:MainObstruction}, and in particular says that even if the Bott filtration is $\E_2$, it will not be $\E_2$-split before smashing with $MU$. 

The proof is via a power operation computation.  In particular, the $\mathbb{A}_\infty$ splitting map takes the stabilization of the bottom cell $\beta_l : S^2 \to  \CP^{n-1} \to \Omega SU(n)$ on the left-hand side to the stabilization of the bottom cell $\beta_r: S^2 \to F_{n,1} \simeq \CP^{n-1}$ on the right-hand side.  We construct a power operation $\nu^{s}$ and show that $\nu^{s}(\Sigma^{\infty} \beta_l) \neq \nu^s(\Sigma^{\infty} \beta_r).$ 

\begin{obs}Let $Y\in \Alg_{\E_2}(\cS)$, and suppose we are given a map $S^2\to Y$.  This extends to an $\E_2$ map $\Omega^2 S^4 \to Y.$  We may precompose with the map $h: S^5 \to \Omega^2 S^4$ adjoint to the Hopf map $S^7\to S^4$.  This procedure determines a natural operation $$\nu^u: \pi_2(Y) \to \pi_5(Y)$$ in the homotopy of any $\E_2$-algebra in spaces.  

Correspondingly, for any $X\in \Alg_{\E_2}(\Sp)$, a class in $\pi_2(X)$ determines an $\E_2$ map $\Sigma^{\infty}_+ \Omega^2 S^4 \to X$.  The above map $h$ then determines an operation $\nu^s :\pi_2(X) \to \pi_5(X)$ via precomposition.  This has the property that for $Y\in \Alg_{\E_2}(\cS)$ and $\beta \in \pi_2(Y)$, we have $\nu^s(\Sigma^{\infty} \beta) = \Sigma^{\infty} \nu^u (\beta).$
\end{obs}

\begin{rmk} \label{rmk:multnu}
The notation is meant to hint at the fact that if $Y = \Omega^\infty X$ comes from a spectrum, then the operation $\nu^u$ is given by multiplication by the element $\nu \in \pi_3(\mathbb{S})$ from the homotopy groups of the sphere spectrum.  Thus, $\nu^u$ is an unstable version of $\nu$ that is already seen in any $\E_2$ algebra in spaces.    
\end{rmk}

We now compute the operation $\nu^s$ on $\Sigma^{\infty} \beta_l$ and $\Sigma^{\infty} \beta_r$.  

\textbf{Computing $\nu^s(\Sigma^{\infty}\beta_l)$:}

For $n>3$, observe that the natural map $\Omega SU(n) \to BU$ is an isomorphism in homology up to degree $7$.  This implies that $\pi_5(\Omega SU(n)) \simeq \pi_5(BU) \cong 0$.  Consequently, $\nu^u(\beta_l) = 0$ and so $\nu^s(\Sigma^{\infty} \beta_l) = 0.$  

\textbf{Computing $\nu^s(\Sigma^{\infty}\beta_r)$:}

For $\beta_r$, we use the assumption that $gr(\Sigma^{\infty}_+ \{ F_{n,k} \})$ is an $\E_2$ graded spectrum.  The map $\beta_r: S^2 \to F_{n,1}$ extends to an $\E_2$ map of underlying $\mathbb{E}_2$-algebras $$\Sigma^{\infty}_+ \Omega^2 S^4 \to  gr(\Sigma^{\infty}_+ \{ F_{n,k} \}).$$  Since $\Sigma^{\infty} \beta_r$ hits the degree 1 piece, this in fact arises as part of a diagram of \emph{graded} spectra:

$$
\begin{tikzcd}
 F_{\E_2}(\Sigma^{\infty}S^2[1])\arrow[rr,"F_{\E_2}(\Sigma^{\infty} \beta_r)"]& & F_{\E_2}(gr(\Sigma^{\infty}_+ \{ F_{n,k}\} ))\arrow[d]\\
 \Sigma^{\infty}S^2[1] \arrow[rr] \arrow[u]&& gr(\Sigma^{\infty}_+ \{ F_{n,k} \} ), \\ 
\end{tikzcd}
$$

where the free algebras are taken in graded spectra (see Example \ref{exm:snaith}).  The relevant map for us is the $\E_2$ map of graded spectra $$F_{\E_2}(\Sigma^\infty S^2[1]) \to gr(\Sigma^{\infty}_+ \{ F_{n,k} \})$$ in that diagram.  In particular, we aim to show that $\nu^s(\Sigma^\infty \beta_r)$ is nonzero. It suffices to see this in grading $1$; there, it is given by the composite $$\Sigma^{\infty} S^5 \to \Sigma^{\infty}_+ \Omega^2 S^4 \to \Sigma^{\infty} S^2 \xrightarrow{\Sigma^{\infty} \beta_r} \Sigma^{\infty} F_{n,1} =\Sigma^{\infty} \CP^{n-1},$$ where the middle map is given by projection onto the first graded piece (i.e., the map from the Snaith splitting).  It is easy to see that the first composite $\Sigma^{\infty} S^5 \to \Sigma^{\infty} S^2$ is simply $\nu \in \pi_3(\mathbb{S}).$  Therefore, the whole composite is given by the product $\nu\cdot (\Sigma^{\infty} \beta_r).$  
However, it was computed in \cite[Theorem II.8]{Liulevicius} that $\pi_5(\Sigma^{\infty}\CP^{\infty})^{\wedge}_2=\mathbb{Z}/2$ generated by $\nu \cdot (\Sigma^{\infty}\beta_r).$  Moreover, the natural map $\Sigma^{\infty}\CP^{n-1} \to \Sigma^{\infty}\CP^\infty$ is an isomorphism on $\pi_5$ for $n>3$.  We conclude that $\nu \cdot (\Sigma^{\infty}\beta_r )\neq 0$ and thus $\nu^s(\Sigma^{\infty} \beta_r) \neq 0$.  This contradicts the existence of an $\E_2$ splitting.

\begin{rmk}
Taking the limit as $n\to\infty$, we see from the above computations that the $\E_2$ power operations on the bottom cells of $BU$ and $Q\CP^{\infty}$ do not agree.  The bottom of the Weiss tower for the functor $V \mapsto BU(V)$ gives a well-known loop map $s:BU \to Q\CP^{\infty}$, implementing the splitting principle.  The obstruction of this section recovers the classical fact that $s$ is not a double loop map.
\end{rmk}

\appendix

\section{Further Properties of Day Convolution}\label{app:day}

\label{sect:monoidal2}

Here we discuss some additional constructions and results that we will need for the more technical parts of this paper.  

The monoidal structures on our categories will arise from Day convolution.  This was studied for $\infty$-categories by Glasman \cite{Glasman} and Lurie \cite{LurieRot, HA} at varying levels of generality.  We will find it convenient to use the formulation from Section 2.2.6 of \cite{HA}.  

\begin{thm}[\cite{HA}, Example 2.2.6.9]
Let $\C$ and $\D$ be symmetric monoidal $\infty$-categories.  Then there is an $\infty$-operad $\Fun(\C, \D)^{\otimes} $ with the following properties:
\begin{enumerate}
\item The underlying $\infty$-category of $\Fun(\C,\D)^{\otimes}$ is the functor category $\Fun(\C, \D)$.
\item The $\infty$-category $\Alg_{\E_\infty}(\Fun(\C, \D)^{\otimes})$ of $\E_\infty$ algebras in $\Fun(\C,\D)^{\otimes}$ is equivalent to the category of lax symmetric monoidal functors from $\C$ to $\D$.  

\end{enumerate}
\end{thm}

In order for the $\infty$-operad $\Fun(\C,\D)^{\otimes}$ to actually be a symmetric monoidal $\infty$-category, one needs to make additional assumptions.  

\begin{prop}[\cite{HA}, Proposition 2.2.6.16]\label{prop:dayconvsmc}
Let $\C$ and $\D$ be symmetric monoidal $\infty$-categories.  Suppose that $\kappa$ is an uncountable regular cardinal such that:
\begin{enumerate}
\item $\C$ is essentially $\kappa$-small.
\item $\D$ admits $\kappa$-small colimits.
\item The tensor product on $\D$ preserves $\kappa$-small colimits separately in each variable.  
\end{enumerate}
Then $\Fun(\C,\D)^{\otimes}$ is a symmetric monoidal $\infty$-category.  
\end{prop}

Recall that the Day convolution is defined classically via left Kan extension.  Assumptions (1) and (2) ensure that the relevant Kan extensions exist.  Assumption (3) then ensures that the multiplication is associative by allowing the colimits taken in the formula for left Kan extension to commute with the tensor product.  

As stated before, Proposition \ref{prop:dayconvsmc} is sufficient to construct symmetric monoidal $\infty$-categories $\Fil(\Sp)$ and $\Gr(\Sp)$.  However, we wish to understand the interaction of the Weiss calculus with multiplicative structure; there, the filtrations go the other way.

We would like to make $\Cofil(\Sp)$ a symmetric monoidal $\infty$-category by putting the Day convolution on its opposite, $\Fun(\Z_{\geq 0}, \Sp^{op}).$  However, the smash product of spectra does not preserve small colimits separately in each variable.  Nevertheless, it does preserve \emph{finite} colimits separately in each variable.  In fact, these are the only colimits that are needed in the case at hand and so we have the following variant of Proposition \ref{prop:dayconvsmc}:

\begin{var}\label{var:day}
Let $\C$ and $\D$ be symmetric monoidal $\infty$-categories.  Suppose that:
\begin{enumerate}
\item Let $I$ be a nonempty finite set and consider the multiplication map $\Pi_{i\in I} \C \to \C$.  For every $C\in \C$, the slice category $\Pi_{i\in I}\C \times_{\C} \C_{/C}$ has a finite cofinal subcategory.  
\item $\D$ admits finite colimits. 
\item The tensor product on $\D$ preserves finite colimits separately in each variable.  
\end{enumerate}
Then $\Fun(\C, \D)^{\otimes}$ is a symmetric monoidal $\infty$-category.  
\end{var}
\begin{proof}
This follows directly from the same arguments as Proposition \ref{prop:dayconvsmc}.  In \cite[Corollary 2.2.6.14]{HA}, the assumptions are used to guarantee the existence of a left Kan extension; this again exists by assumptions (1) and (2) and \cite[Lemma 4.3.2.13]{HTT}.  Similarly, the proof of \cite[Proposition 2.2.6.16]{HA} only makes reference to commuting tensor products in $\D$ with finite colimits, which is ensured by assumption (3).  
\end{proof}

In Section \ref{app:SplittingMachine}, we will need to consider not only the Day convolution monoidal structure on $\Fun(\C,\D)$ but its functoriality as $\C$ varies.  For instance, we would for symmetric monoidal functors $\C_1 \to \C_2$ to induce symmetric monoidal functors $\Fun(\C_1,\D) \to \Fun(\C_2,\D)$ via left Kan extension.  

We give a very close variant of \cite[Corollary 3.8]{Nikolaus} in our current framework:

\begin{prop}\label{prop:kanmonoidal}
Let $\C_1$, $\C_2$, and $\D$ be symmetric monoidal $\infty$-categories and let $f:\C_1 \to \C_2$ be a symmetric monoidal functor.  Suppose that one of the following conditions hold:
\begin{enumerate}
\item The pairs $(\C_1, \D)$ and $(\C_2, \D)$ satisfy the hypotheses of Proposition \ref{prop:dayconvsmc}
\item The pairs $(\C_1, \D)$ and $(\C_2, \D)$ satisfy the hypotheses of Variant \ref{var:day} and for any object $c\in \C_2$, the slice category $\C_1 \times_{\C_2} {\C_2}_{/c}$ has a finite cofinal subset.  
\end{enumerate}
   Then there is an adjunction 
$$ f_! : \Fun(\C_1, \D) \xrightleftharpoons{\quad} \Fun(\C_2, \D): f_* $$ 
where $f_*$ denotes restriction and $f_!$ denotes left Kan extension.  Moreover, the functor $f_*$ is lax symmetric monoidal and $f_!$ is symmetric monoidal.  
\end{prop}
\begin{proof}
The universal property of $\Fun(\C_1, \D)^{\otimes}$ immediately implies the existence of a map of $\infty$-operads $\Fun(\C_2,\D)^{\otimes} \to \Fun(\C_1,\D)^{\otimes}$, which makes $f_*$ a lax symmetric monoidal functor.  

Assumptions (1) and (2) of Proposition \ref{prop:dayconvsmc} guarantee that the adjunction exists at the level of $\infty$-categories.  The rest of the proof from \cite[Corollary 3.8]{Nikolaus} carries over verbatim.  
\end{proof}

\subsection{Proof of Lemma \ref{lem:cubes}}
Let $A_n\subset \Z^{op}_{\geq 0} \times \Z^{op}_{\geq 0} $ be the full subcategory spanned by pairs $(p,q)$ with $p+q \leq n$.  

Define a functor $T:\Z^{op}_{\geq 0} \times \Z^{op}_{\geq 0}  \to \Sp^{\J}$ by the formula $T(p,q) = X_p \wedge Y_q.$  We have by definition that $\lim T|_{A_n} \simeq (X\otimes Y)_{n}$.

Define the functor $\overline{T_n}: A_n \to \Sp^{\J}$ as the right Kan extension of $T|_{A_{n-1}}$ along the inclusion $A_{n-1} \to A_{n}.$   Then, $\overline{T_n}$ has the following properties:
\begin{enumerate}
\item $\lim \overline{T_n} = \lim T|_{A_{n-1}}.$
\item $\overline{T_n}|_{A_{n-1}} = T|_{A_{n-1}}.$
\item $\overline{T_n}(n,0) = X_{n-1} \wedge Y_0$ and $\overline{T_n}(0,n) = X_0\wedge Y_{n-1}.$
\item $\overline{T_n}(p,q) = X_{p-1}\wedge Y_{q} \times_{X_{p-1}\wedge Y_{q-1}} X_p\wedge Y_{q-1}$ for $p+q=n$, $p,q\geq 1$.  
\end{enumerate}

We may therefore compute $$\text{fib}\big(\lim_{A_n} T|_{A_n} \to \lim_{A_{n-1}} T|_{A_{n-1}}\big) = \text{fib}\big(\lim_{A_n} T|_{A_n} \to \lim_{A_n} \overline{T_n}\big) = \lim_{A_n} \text{fib}\big(T|_{A_n} \to \overline{T_n}\big).$$

The conclusion now follows immediately from the usual fact that $$\text{fib}(X_p\wedge Y_q \to X_{p-1}\wedge Y_{q} \times_{X_{p-1}\wedge Y_{q-1}} X_p\wedge Y_{q-1}) = \text{fib}(X_p\to X_{p-1}) \wedge \text{fib}(Y_q \to Y_{q-1}).$$


\section{Proof of Theorem \ref{thm:SplitMachine}}\label{app:SplittingMachine}



We will need a few preliminary definitions.  We start by fixing a positive integer $n$.  
Let $[n]$ denote the linearly ordered set of integers $0\leq i\leq n$. For any indexing 1-category $\mathcal{D}$, denote by $\mathcal{D}^{ds}$ the underlying discrete category, and denote by $\mathcal{D}^+$ the the category formed by formally adding a final object, which we will refer to as ``$+$''.  Define $\Fil_n^+ = \Fun([n]^+, \Sp)$ and $\Cofil_n^+ = \Fun(([n]^+)^{op},\Sp).$  These categories admit functors to $\Sp$ by restriction to the distinguished point.  We define $\C_n$ by the following pullback:

\begin{equation}\label{dia:fundpb}
\begin{tikzcd}
\C_n \arrow[r] \arrow[d]&  \Cofil_n^+ \arrow[d]\\
\Fil_n^+ \arrow[r]& \Sp
\end{tikzcd}
\end{equation}

An element of $\C_n$ can be thought of as a sequence of spectra connected by maps:
\begin{center}
$X_0 \longrightarrow X_1 \longrightarrow \cdots \longrightarrow X_n \longrightarrow X \simeq Y \longrightarrow Y_n \longrightarrow \cdots \longrightarrow Y_1 \longrightarrow Y_0$
\end{center}
where the middle arrow is an equivalence, as indicated.  This is equivalent to just considering sequences
\begin{center}
$X_0 \longrightarrow X_1 \longrightarrow \cdots \longrightarrow X_n \longrightarrow Z \longrightarrow Y_n \longrightarrow \cdots \longrightarrow Y_1 \longrightarrow Y_0$,
\end{center}
and so we shall refer to general elements by these names below.

Define the subcategory $\mathcal{G}_n\subset \C_n$ as the full subcategory such that for each integer $0\leq i\leq n$, the composite $X_i \longrightarrow Y_i$ is an equivalence.  

\begin{lem}
There is an equivalence $$\mathcal{G}_n \simeq \Gr_n^+ := \Fun(([n]^{+})^{ds}, \Sp).$$
\end{lem}
\begin{proof}
We proceed by induction on $n$.  

For $n=0$, we are considering the full subcategory of diagrams $X_0 \to Z \to Y_0$ of spectra with the property that the composite is an equivalence.  By taking the fiber of the second map, this is equivalent to the category of triples $(X_0, Y_0',Z)$ of spectra together with an equivalence $X_0 \vee Y_0' \xrightarrow{\sim} Z.$   This is  certainly equivalent to the category of pairs $(X_0, Y_0')$ of spectra, which is $\Gr_0^+.$  

Next, assume the statement for $n\leq k$ and consider $\mathcal{G}_{k+1}.$  We consider the auxiliary category $\overline{\mathcal{G}}_{k+1}$ which is the full subcategory of $\C_{k+1}$ where only $X_{k+1} \to Y_{k+1}$ is stipulated to be an equivalence.    The argument for the base case shows that $$\overline{\mathcal{G}}_{k+1} = \Fil_{k+1} \times_{\Sp} \Gr_0^+ \times_{\Sp} \Cofil_{k+1}$$ where the fiber products are over the restriction to $0\in [0]^+$ for $\Gr_0^+$, and over $X_{k+1}$ and $Y_{k+1}$ in the filtered and cofiltered spectra.  By commuting the fiber products, we find that $$\overline{\mathcal{G}}_{k+1} = (\Fil_{k+1} \times_{\Sp} \Cofil_{k+1}) \times_{\Sp} \Gr_0^+  \simeq \C_k \times_{\Sp} \Gr_0^+$$ where we have implicitly used the identifications $\Fil_{k+1} \simeq \Fil_k^+$ and $\Cofil_{k+1} \simeq \Cofil_k^+.$   Under this equivalence, the full subcategory $\mathcal{G}_{k+1} \subset \overline{\mathcal{G}}_{k+1}$ corresponds to $\mathcal{G}_k \times_{\Sp} \Gr_0^+ \simeq \Gr_{k+1}^+$ as desired.  

\end{proof}

In fact, the functor $\Gr_n^+ \to \C_n$ can be seen very explicitly as follows: there's a functor $$I_n^+: \Gr_n^+ \to \Fil_n^+$$ given by left Kan extension along the inclusion $([n]^+)^{ds} \to [n]^+$ which is completely analogous to the functor $I$ described in Section \ref{sec:FilGra}.  Dually, there's a functor $$I_n^{op,+}:\Gr_n^+ \to \Cofil_n^+$$ given by right Kan extension along the inclusion $([n]^+)^{ds} \to ([n]^+)^{op}$ which sends an element $(X_0, X_1, \cdots, X_n, X)\in \Gr_n^+$ to $$X_0 \longleftarrow X_0\vee X_1 \longleftarrow \cdots \longleftarrow \bigvee_i X_i \longleftarrow X \vee \bigvee_i X_i.$$  These functors agree on restriction to the distinguished object, and so they define the desired functor $\Gr_n^+ \to \C_n.$


Until this point, we have been working with a fixed $n$ and without regard to the monoidal structure.  The results of Appendix \ref{app:day} allow us to analyze what happens as $n$ varies.   

In particular, give $[n]^+$ the structure of a symmetric monoidal category by taking $\Z_{\geq 0}$ under addition and identifying all the integers $m >n$ with the point $+$.  By Proposition \ref{prop:dayconvsmc}, this gives $\Fil_n^+$ the structure of a symmetric monoidal $\infty$-category.  There are natural symmetric monoidal functors $\Z_{\geq 0}^+ \to [n+1]^+ \to [n]^+$ by successive quotient.  By Proposition \ref{prop:kanmonoidal} combined with the commutative diagram of symmetric monoidal functors

$$
\begin{tikzcd}
(\lbrack n+1\rbrack^+)^{ds} \arrow[r] \arrow[d] & \lbrack n+1\rbrack^+ \arrow[d]\\
(\lbrack n\rbrack^+)^{ds} \arrow[r] & \lbrack n\rbrack^+ ,
\end{tikzcd}
$$
we obtain a commutative diagram of symmetric monoidal $\infty$-categories:
$$
\begin{tikzcd}
\Gr_{n+1}^+ \arrow[r,"I_{n+1}^+"] \arrow[d] & \Fil_{n+1}^+ \arrow[d]\\
\Gr_n^+  \arrow[r,"I_n^+"] & \Fil_n^+ .
\end{tikzcd}
$$
We claim that the limit of the right vertical arrows over $n$ is $\Fil^+ := \Fun(\Z_{\geq 0}^+,\Sp)$ with its symmetric monoidal structure as given by Proposition \ref{prop:dayconvsmc}.  This is because by Proposition \ref{prop:kanmonoidal}, there are symmetric monoidal functors $\Fil^+ \to \Fil_n^+$ induced by the symmetric monoidal functor $\Z_{\geq 0}^+ \to [n]^+$ collapsing $+$ and all integers $m> n$ to a point.  This induces a symmetric monoidal functor $\Fil^+ \to \lim_n \Fil_n^{+}.$  It is then easy to check that this is an equivalence.  

There is a similar diagram for the cofiltered side by the appropriate analogs of the above statements:
$$
\begin{tikzcd}
\Gr_{n+1}^+ \arrow[r,"I_{n+1}^{op,+}"] \arrow[d] & \Cofil_{n+1}^+ \arrow[d]\\
\Gr_n^+  \arrow[r,"I_n^{op,+}"] & \Cofil_n^+ 
\end{tikzcd}
$$
and as before, we have an identification $ \Fun((\Z_{\geq 0}^{+})^{op}, \Sp) =: \Cofil^+ \simeq \lim_n \Cofil_n^+$ by Variant \ref{var:day} and Proposition \ref{prop:kanmonoidal}.

Thus, taking the limit in $n$ in the diagram (\ref{dia:fundpb}) yields a diagram of symmetric monoidal functors:

\begin{equation}\label{dia:fund}
\begin{tikzcd}
\mathcal{G}_\infty
 \arrow[drr, bend left, "I^{+,op}"]
  \arrow[ddr, bend right, "I^+"]
  \arrow[dr] & & \\
&\C_\infty \arrow[r] \arrow[d]&  \Cofil^+ \arrow[d]\\
& \Fil^+ \arrow[r]& \Sp
\end{tikzcd}
\end{equation}
where the square is Cartesian.

\begin{rmk}
While $\mathcal{G}_\infty \simeq \lim \Gr_n^{+}$, we warn the reader that $\mathcal{G}_\infty$ is not simply $\Fun((\Z_{\geq 0}^+)^{ds}, \Sp)$ because the maps in the inverse system for $\mathcal{G}_\infty$ are not just the ones induced by the inclusions $([n]^+)^{ds} \to ([n+1]^+)^{ds}.$
\end{rmk}

The functor $\mathcal{G}_\infty \to \mathcal{C}_\infty$ is fully faithful because it is the inverse limit of fully faithful functors.  Moreover, the essential image consists of those pairs $(X,Y)\in \Fil^+\times_{\Sp} \Cofil^+ = \C_{\infty}$ such that the natural maps $X_i\to Y_i$ are equivalences for all $i\geq 0$.   Since all the functors in the diagram (\ref{dia:fund}) were symmetric monoidal, we have for each integer $n\geq 0$ a diagram at the level of algebras:

\begin{equation}\label{dia:fundmon}
\begin{tikzcd}
\Alg_{\E_n}(\mathcal{G}_\infty)
 \arrow[drr, bend left, "I^{+,op}"]
  \arrow[ddr, bend right, "I^+"]
  \arrow[dr] & & \\
&\Alg_{\E_n}(\C_\infty) \arrow[r] \arrow[d]&  \Alg_{\E_n}(\Cofil^+ )\arrow[d]\\
& \Alg_{\E_n}(\Fil^+ )\arrow[r]& \Alg_{\E_n}(\Sp)
\end{tikzcd}
\end{equation}
where the square is Cartesian and the functor $\Alg_{\E_n}(\mathcal{G}_\infty) \to \Alg_{\E_n}(\C_\infty)$ is fully faithful with essential image as described above.

Recall that we were interested in understanding when an $\E_n$ filtered spectrum $X\in \Alg_{\E_n}(\Fil)$ is split - that is, when there exists $Z\in \Alg_{\E_n}(\Gr)$ such that $X \simeq IZ.$  The following proposition relates that to our current situation; informally, it allows us to get rid of the +'s.

\begin{prop}\label{prop:MonRet}
There exists a diagram of symmetric monoidal $\infty$-categories and symmetric monoidal functors 
$$
\begin{tikzcd}
&\mathcal{G}_\infty \arrow[r,"\pi"] \arrow[d, "I^+"]&  \Gr \arrow[d, "I"]\\
\Fil \arrow[r,"\iota"] & \Fil^+ \arrow[r,"\varpi"]& \Fil
\end{tikzcd}
$$
where the bottom row is a retract and $I^+$ is induced by the $I_n^+$ at each finite level.  
\end{prop} 
\begin{proof}

We have the square at each finite level: 
$$
\begin{tikzcd}
\Gr_{n}^+ \arrow[r] \arrow[d,"I_n^{+}"] & \Gr_{n} \arrow[d,"I_n"]\\
\Fil_n^+  \arrow[r] & \Fil_n .
\end{tikzcd}
$$
where the right arrow comes from ignoring the $+$.  It is easy to see that all the functors are symmetric monoidal, and so taking the limit gives the square.  

For the bottom row, we apply Proposition \ref{prop:kanmonoidal} to the symmetric monoidal functor $q:\Z_{\geq 0} \to \Z_{\geq 0}^+.$  The functor $\varpi$ certainly coincides with $q_*$, and we define $\iota = q_!$, which is symmetric monoidal by Proposition \ref{prop:kanmonoidal}.  It is immediate that $\varpi \circ \iota$ is an equivalence, so the bottom row is a retract.  

\end{proof}


We are now ready to prove the main result of this section.

\begin{proof}[Proof of Theorem \ref{thm:SplitMachine}]
We use the notations of Proposition \ref{prop:MonRet}.  Since $\iota$ is a symmetric monoidal functor, we obtain an $\E_n$ algebra $\iota X \in \Alg_{\E_n}(\Fil^+).$  
On the cofiltered side, we note that by the universal property of Day convolution, the symmetric monoidal functor $\Z_{\geq 0} \to \Z_{\geq 0}^+$ induces a map of $\infty$-operads $\Fun(\Z_{\geq 0}^+, \Sp^{op})^{\otimes} \to \Fun(\Z_{\geq 0},\Sp^{op})^{\otimes}.$  This amounts to an oplax symmetric monoidal functor $\Cofil^+\to \Cofil$ which restricts away from the $+$.  It is easy to see that this is actually a symmetric monoidal functor, and so its adjoint $\iota^{op}: \Cofil \to \Cofil^+$ by right Kan extension is lax monoidal.  Concretely, this is the functor which takes a cofiltered spectrum and adds in its limit.  As a result, we obtain an $\E_n$-algebra $\iota^{op}Y \in \Alg_{\E_n}(\Cofil^+)$ with image $\lim Y\in \Alg_{\E_n}(\Sp).$  

Condition (1) in the statement of the theorem guarantees that $\iota X$ and $\iota^{op}Y$ determine an element $\mathcal{X} \in \Alg_{\E_n}(\C_{\infty})$.  Condition (2) then ensures that $\mathcal{X}$ is in the essential image of the fully faithful functor $\Alg_{\E_n}(\mathcal{G}_\infty) \to \Alg_{\E_n}(\C_\infty)$; consequently, we regard $\mathcal{X}$ as an $\E_n$-algebra in $\mathcal{G}_{\infty}$.  Finally, we chase through the diagram of Proposition \ref{prop:MonRet} to see that $I\pi \mathcal{X} \simeq \varpi I^+ \mathcal{X} \simeq \varpi \iota X \simeq X$ as $\E_n$-algebras in $\Fil$.   Hence, $X$ is $\E_n$-split, as desired. 
\end{proof}

\bibliographystyle{alpha}
\bibliography{Bibliography}

\end{document}